\documentclass[12pt]{amsart}
\usepackage[backend=biber]{biblatex}
\usepackage{amssymb}
\usepackage{amsmath,amsthm,appendix}
\usepackage{verbatim}
\usepackage{graphicx}
\usepackage{epstopdf}
\usepackage{epsf}
\usepackage{geometry}
\usepackage[UKenglish]{datetime}
\usepackage[colorlinks=true, linkcolor=blue, filecolor=magenta, urlcolor=cyan, citecolor=gray, unicode]{hyperref}
\usepackage[flushmargin]{footmisc}
\usepackage{subfig}
\usepackage[usenames,dvipsnames]{color}
\usepackage[table,xcdraw]{xcolor}
\usepackage{diagbox}

\usepackage{enumitem}
\setlist{itemsep=0pt, topsep=0pt}

\usepackage[all]{xy}
\usepackage{tikz-cd}
\usepackage{mathtools}

\usepackage{multicol,setspace, multirow}% http://ctan.org/pkg/{enumitem,multicol,setspace}

\usepackage{slashbox}

\newtheorem{theorem}{Theorem}[section]
\newtheorem{lemma}[theorem]{Lemma}
\newtheorem{claim}{Claim}

\newtheorem{conjecture}[theorem]{Conjecture}
\newtheorem{corollary}[theorem]{Corollary}

\newtheorem{proposition}[theorem]{Proposition}

\newtheorem{problem}[theorem]{Problem}
\theoremstyle{definition}

\newtheorem{case}{Case}

\numberwithin{subcase}{case}

\newtheorem{definition}[theorem]{Definition}

\numberwithin{property}{theorem}

\newcommand{\ep}{\epsilon}

\newcommand{\floor}[1]{\left\lfloor#1\right\rfloor}
\newcommand{\ceiling}[1]{\left\lceil#1\right\rceil}

\newcommand{\cL}{{\mathcal L}}
\newcommand{\cH}{{\mathcal H}}
\newcommand{\cM}{{\mathcal M}}

\newcommand{\cS}{{\mathcal S}}

\newcommand{\vB}{\vec{B}}
\newcommand{\vC}{\vec{C}}
\newcommand{\vE}{\vec{E}}
\newcommand{\vF}{\vec{F}}
\newcommand{\vG}{\vec{G}}
\newcommand{\vH}{\vec{H}}
\newcommand{\vK}{\vec{K}}
\newcommand{\vM}{\vec{M}}

\newcommand{\vS}{\vec{S}}
\newcommand{\vT}{\vec{T}}
\newcommand{\vW}{\vec{W}}

\newcommand{\vcS}{\vec{\cS}}

\newcommand{\ex}{\mathrm{ex}}

\newcommand{\defeq}{\vcentcolon=}

\title{Orientations of graphs omitting non-edge-critical directed graphs}
\author{H. Sheats}
\date{\today}
\thanks{The author was partially supported by NSF Award DMS-2115518}
\address{Department of Mathematics, Statistics, and Computer Science, University of Illinois Chicago}
\email{hshea3.uic.edu}
\begin{document}

\begin{abstract}
   In 1974, Erd\H{o}s asked the following question: given a graph $G$ and a directed graph $\vH$, how many ways are there to orient the edges of $G$ such that it does not contain $\vH$ as a subgraph? We denote this value by $D(G, \vH)$. Further, we let $D(n, \vH)$ denote the maximum of $D(G, \vH)$ over all graphs $G$ on $n$ vertices. In 2006, Alon and Yuster gave an exact answer when $\vH$ is a tournament. In 2023, Buci\'c, Janzer, and Sudakov gave asymptotic answers for all directed graphs $\vH$, and in the same paper, they gave an exact answer when $\vH$ is a directed cycle. In this paper, we give a better bound for some specific non-bipartite directed graphs. Further, we obtain exact values of $D(G, \vH)$ for some small non-edge-critical directed graphs $\vH$. Finally, for these graphs, we classify all graphs $G$ that attain the bound $D(G, \vH) = D(n, \vH)$.
\end{abstract}

\maketitle
\section{Introduction}
Given a directed graph $\vec{H}$ and an (undirected) graph $G$, how many ways can one orient the edges of $G$ to form a directed graph $\vec{G}$ which does not contain a copy of $\vec{H}$?  This question was first posed by Erd\H{o}s in 1974 \cite{Erdos1974}. We call a directed graph \textit{$\vH$-free} if it does not contain $\vH$ as a subgraph.  Let $D(G, \vH)$ denote the number of $\vH$-free orientations of the edges a graph $G$. Further, let $D(n, \vH)$ denote the maximum of $D(G, \vH)$ over all graphs $G$ on $n$ vertices. Erd\H{o}s's question then asks what is $D(n, \vH)$, and in particular, what graphs $G$ satisfy $D(G, \vH) = D(n, \vH)$.

Let $\vH$ be a directed graph. We denote by $H$ the underlying graph of $\vH$, i.e., the graph formed by ``forgetting" the orientations of the edges of $\vH$.  An easy lower bound for $D(n, \vH)$ can be obtained by observing that if a graph $G$ does not contain $H$ as a subgraph, then any orientation $\vG$ of $G$ does not contain the directed graph $\vH$ as a subgraph. Thus, the number of $\vH$-free orientations of $G$ is  $2^{|E(G)|}$, the total number of orientations of the edges of $G$. When we take $G$ to be an extremal graph for $H$, this gives us $D(n, \vH) \geq D(G, \vH) \geq 2^{\ex(n, H)}$. This bound is known to be tight in some but not all cases, as we will see below. Our main problem of study is as follows:

 \begin{problem}
 For which directed graphs $\vH$ is $D(n, \vH) = 2^{\ex(n, H)}$?
 \end{problem}

A \textit{tournament} is a directed graph $\vK_r$ such that there is exactly one directed edge between any two vertices. The first major progress on this problem was Alon and Yuster, in 2006, who, in \cite{Alon2006}, showed the following:

\begin{theorem}[Alon, Yuster \cite{Alon2006}]
    If $\vK_r$ is a tournament on $r$ vertices, then $D(n, \vK_r) = 2^{\ex(n, K_r)}$ for large enough $n$.
\end{theorem}
More recently in \cite{botler2021}, Botler, Hoppen and Mota proved that if $\vcS_k$ is a the family of strongly connected tournaments on $k$ vertices, then $D(n, \vcS_k) - 2^{\ex(n, K_k)}$.

In 2023, Buci\'c, Janzer, and Sudakov \cite{bucic2023} proved the following result for odd cycles:

\begin{theorem}[Buci\'c, Janzer, Sudakov \cite{bucic2023}]
    If $\vC$ is an odd cycle directed in the natural way, i.e. as a directed cycle, then, for large enough $n$,
    $$D(n, \vC) = 2^{\ex(n, C)}.$$
\end{theorem}
In the same paper \cite{bucic2023}, Buci\'c, Janzer, and Sudakov resolved the question asymptotically for all graphs.

As is noted in \cite{bucic2023}, with slight modifications, the method used to obtain the result for odd cycles can be extended to show that for any orientation $\vH$ of an edge-critical graph $H$ with chromatic number at least 3, $D(n, \vH) = 2^{\ex(n, H)}$.  We give an explicit proof of this fact as a special case of one of our results, Theorem \ref{general_graph_theorem}, below. 

A main goal of this paper is to begin the study of the function $D(n, \vH)$ in the case of some nice orientations of non-edge-critical graphs.

Before introducing our main tools, we introduce a few pieces of notation. For a family of graphs $\cL$, we will write $p = \min\{\chi(L):L\in \cL\}-1$. Further, if $G, H$ are graphs, we denote by $G\otimes H$ the graph consisting of disjoint copies of $G$ and $H$, with edges between every pair of vertices $\{u, v\}$ where $u\in V(G)$ and $v\in V(H)$. We also let $K_{p-1}(t)$ denote the complete $(p-1)$-partite graph with $t$ vertices in each part. In \cite{BALOGH20041}, Balogh, Bollob\'as and Simonovitz gave the following definition: 
\begin{definition}[Balogh, Bollob\'as, Simonovitz \cite{BALOGH20041}]
\label{original_m_def}
    Given a family of graphs $\cL$, the \emph{decomposition family} $\cM(\cL)$ is the family of minimal graphs $M$ such that for some $t$, the graph $M'\otimes K_{p-1}(t)$ contains a graph $L\in \cL$ as a subgraph, where $M'$ denotes the graph $M$ with $t$ isolated vertices adjoined to it.
\end{definition}
In other words, $\cM$ is the family of graphs $M$ such that if you place $M$ in a single part of $K_{p}(t)$ then the resulting graph contains a graph from $\cL$. This has been useful in a variety of contexts.  We aim to find a suitable analogue of $\cM$ for directed graphs. Our first, naive attempt is as follows:
\begin{definition}
\label{M_def}
     Given a directed graph $\vH$, let $\cM(\vH)$ denote the family directed graphs $\vM$ whose underlying graph $M$ is in $\cM(H)$ and such that for some $t\geq 1$, the following hold, where $M'$ is $M$ adjoined to $t$ isolated vertices: There exists an orientation $\vG_M$ of $G_M = M'\otimes K_{p-1}(t)$, and a set $W$ satisfying $V(M)\subseteq W\subseteq V(G_M)$ such that $\vH\subseteq \vG_M[W]$ and $\vG_M[M]\cong\vM$.
\end{definition}
Note that for every $M\in \mathcal{M}(H)$, $\mathcal{M}(\vH)$ contains at least one orientation $\vM$ of $M$.  Although Definition \ref{M_def} will prove useful in simplifying notation, we will need something stronger to prove our results:

Let $\cH$ be a family of graphs.  We denote by $\overrightarrow{\cH}$ the family of all orientations of graphs $H\in \cH$.
\begin{definition}
\label{M'_def}
    Let $\cM'(\vH) = \{M\in \cM(H): \overrightarrow{\{M\}}\subseteq \cM(\vH)\}$. That is, $\cM'(\vH)$ is the set of graphs $M\in \cM(H)$ such that every orientation of $M$ is in $\cM(\vH)$.
\end{definition}

Notice that $\cM(\vH)$ is a family of directed graphs. In contrast, $\cM'(\vH)$ is a family of (undirected) graphs. Indeed, $\cM'(\vH)\subseteq \cM(H)$, and $\cM(\vH)\subseteq \overrightarrow{\cM(H)}$.  Our main theorem will make use of the family $\cM'(\vH)$.

Before stating our results, we need some more terminology. A \textit{$(t+1)$-star}, denoted $S_{t+1}$, is the graph with a single center vertex $v$ and $t$ other vertices, all adjacent only to $v$. We call an orientation $\vS_{t+1}$ of a star $S_{t+1}$ \textit{pure} if all of its edges are oriented towards its center vertex or all of its edges are oriented away from its center vertex. Note $S_{t+1}$ has exactly two pure orientations. We now state our first main result.

\begin{theorem}
\label{general_graph_theorem}
    Let $H$ be a graph with chromatic number $\chi(H) = p+1 \geq 3$, and suppose there is a $(t+1)$-star $S_{t+1}$ in $\cM(H)$. Fix an orientation $\vH$ of $H$ and suppose $\cM(\vH)$ contains both pure orientations of $S_{t+1}$. Then for large enough $n$,
    $$2^{\ex(n, H)}\leq D(n,\vH)\leq 2^{t(n,p)+p\ex(n,\cM'(\vH))}.$$
\end{theorem}

Note that if $H$ is an edge critical graph, then for any orientation $\vH$ of $H$, $\cM'(\vH) = \cM(H)$ which is a single edge. Thus we immediately obtain the result from \cite{bucic2023} that $D(n, \vH) = 2^{\ex(n, H)}$. 

The proof of Theorem \ref{general_graph_theorem}, gives no information about the structure of a graph $G$ maximizing $D(n, \vH)$ for a directed graph $\vH$. Our next theorem shows that in the case where $\ex(\cM'(\vH))$ is a constant value, for large enought $n$, one can change a constant number of edges in $G$ to obtain a complete $p$-partite graph.

\begin{theorem}
\label{general_gph_finite_thm}
    Let $H$ be a graph with chromatic number $\chi(H) = p+1 \geq 3$ and let $\vH$ be an orientation of $H$. Suppose that there is a $(t+1)$-star $S_{t+1}$ in $ \cM(H)$ and that $\cM(\vH)$ contains both pure orientations of $S_{t+1}$. Suppose further that there is a constant $a$ such that for large enough $n$, $\ex(n, \cM'(\vH))= a$. Let $G$ be a graph that maximizes $D(G, \vH)$. If $V(G) = V_1\cup \hdots \cup V_p$ is an optimal $p$-partition of $V(G)$, then 
    $$\sum_{i=1}^p|E(V_i)|\leq pa.$$
\end{theorem}
In particular, this immediately implies the following corollary.
\begin{corollary}
    Let $H$ be an edge critical graph with chromatic number $\chi(H) = p+1\geq 3$, and let $\vH$ be any orientation of $H$. Then for large enough $n$,
    $$D(n, \vH) = 2^{\ex(n, H)} = 2^{t(n, p)}.$$
    Furthermore, $T(n, p)$ is the unique graph attaining this bound.
\end{corollary}

In Section 4 we give an application of Theorem \ref{general_gph_finite_thm}.  In particular, we will see that Theorems \ref{general_graph_theorem} and \ref{general_gph_finite_thm} are especially interesting when $\cM'(\vH) = \overrightarrow{\cM(H)}$. We define a $(k, r)$-\emph{fan}, denoted $F_{k, r}$ to be $k$ copies of $K_r$ all joined at a single vertex. We call the common vertex to all copies of $K_r$ the \textit{center} of $F_{k, r}$.  The following result about the extremal graphs for $(k, r)$-fans was shown for $r = 3$ by Erd\H{o}s, F{\"u}redi, Gould, and Gunderson \cite{Erds1995ExtremalGF} and for general $r$ by Chen \cite{CHEN2003159}.
\begin{theorem}(Chen \cite{CHEN2003159})
\label{extremal_thm_kfans}
    For every $k\geq 1$, $r\geq 2$ and for every $n\geq 16 k^3r^8$ if a graph $G$ on $n$ vertices has more than
    $$t(n, r-1)+\begin{cases}
    k^2-k & \text{ if $k$ is odd}\\
    k^2 -\frac{3}{2}k & \text{ if $k$ is even}
    \end{cases}$$
    edges, then $G$ contains a copy of a $(k, r)$-fan.  Furthermore, for any $n$, there exists a $(k, r)$-fan-free graph with this many edges.
    \end{theorem}

For ease of notation, from here on we denote 
$$h(k) =\begin{cases}
    k^2-k & \text{ if $k$ is odd}\\
    k^2 -\frac{3}{2}k & \text{ if $k$ is even}
    \end{cases} $$
so that Theorem \ref{extremal_thm_kfans} says
$$\ex(n,F_{k,r}) = t(n, r-1)+h(k).$$
One may deduce from results in \cite{CHEN2003159} that $\cM(F_{k, r}) = \{S_{k+1}, M_{k}\}$. That is, $\cM(F_{k, r})$ consists of a $(k+1)$-star and a matching of size $k$.  We define a specific family of orientations of $F_{k,r}$ that satisfy the conditions of Theorems \ref{general_graph_theorem}, and \ref{general_gph_finite_thm}:

\begin{definition}
    Let $\vF_{k,r}$ be an orientation of $F_{k, r}$ such that each copy of $K_r$ has at least one edge entering the center vertex and at least one edge exiting the center vertex. We call such an orientation \emph{anti-directed}.
\end{definition}

One can easily check that if $\vF_{k, r}$ is an anti-directed orientation of $F_{k, r}$,  then $\cM'(\vF_{k, r}) = \overrightarrow{\cM(F_{k, r})}$. Thus, combining Theorem \ref{general_graph_theorem} with the results in \cite{CHEN2003159}, we obtain the following corollary:
\begin{corollary}
    \label{kfans_cor}
     Let $k\geq 3, r\geq 2$ . Let $\vF_{k, r}$ be an anti-directed orientation of $F_{k, r}$. Then for large enough $n$, we have
    $$2^{t(n, r-1)+h(k)}\leq D(n,\vF_{k, r})\leq 2^{t(n, r-1)+(r-1)h(k)}.$$
\end{corollary}
We will then use Corollary \ref{kfans_cor} to prove exact results for some small cases of $k, r$:

\begin{theorem}
\label{k-fans_thm_1}
    Let $r\geq 2$ and let $\vF_{2, r}$ be an anti-directed orientation of $F_{2,r}$. Then for large enough $n$,
    $$D(n,\vF_{2, r}) =2^{\ex(n, F_{2, r})} = 2^{t(n, r-1)+1}.$$
\end{theorem}
\begin{theorem}
    \label{k-fans_thm_2}
    Let $\vF_{3, 3}$ be an anti-directed orientation of $F_{3,3}$. Then for large enough $n$, 
    $$D(n, \vF_{3, 3}) = 2^{\ex(n, F_{3,3})}= 2^{\floor{\frac{n^2}{4}}+6}.$$
\end{theorem}

In Section 5 we discuss the remaining orientations of $F_{2,3}$ which we call a \emph{bowtie}.  In particular, we will see that there are some orientations $\vF_{2,3}$ such that $D(n, \vF_{2,3})> 2^{\ex(n, F_{2,3})}$:
\begin{proposition}
\label{not_semi_anti_directed_prop}
Let $\vB$ be an orientation of $B$ such that either:
\begin{enumerate}[label=(\alph*)]
    \item All edges incident to the center of $\vB$ are entering the center vertex or all edges incident to the center are exiting the center vertex.
    \item One triangle of $\vB$ has both edges entering the center and the other triangle has both edges exiting the center.
\end{enumerate}
Then $D(n, \vB)>2^{\ex(n, B)}$.
\end{proposition}
We then give the following example:
\begin{proposition}
    \label{wheel_prop}
    There exists a graph $H$ such that for any orientation $\vH$ of $H$
    $$D(n, \vH) > 2^{\ex(n, H)}.$$
\end{proposition}
In the remainder of this section we present some questions and open problems.

\subsection{Questions and open problems}
We begin with some conjectures about the unresolved orientations of $B = F_{2,3}$. First, we conjecture the value of $D(n, \vB)$ for the orientations mentioned in Proposition \ref{not_semi_anti_directed_prop}.

\begin{conjecture}
\label{all_bowties_conj}
    Let $G$ be a complete tripartite graph with part sizes $1, \floor{\frac{n-1}{2}}$, and $\ceiling{\frac{n-1}{2}}$. Let $\vB$ be an orientation of $B$ as in Case (a) or (b) of Proposition \ref{not_semi_anti_directed_prop}. Then
    $$D(G, \vB) = D(n, \vB).$$
\end{conjecture}
Now, we have the remaining case:
\begin{conjecture}
    Let $\vB$ be an orientation of $B$ such that there are exactly three edges oriented towards the center vertex (or exactly three edges oriented away from the center vertex). Then
    $$D(n, \vB) = 2^{\ex(n, B)}$$
\end{conjecture}

Next, we have a very general and difficult problem.

\begin{problem}
\label{classification_prob}
For which graphs $H$ is there an orientation $\vH$ such that 
$$D(n, \vH) = 2^{\ex(n, H)}?$$
\end{problem}

As a start to Problem \ref{classification_prob}, and in light of the results of Theorems \ref{k-fans_thm_1} and \ref{k-fans_thm_2}, it seems reasonable to conjecture the following:

    \begin{conjecture}
        Let $\vH$ be a directed graph with underlying graph $H$ such that $\cM(\vH) = \overrightarrow{\cM(H)}$. Then for large enough $n$,
        $$D(n, \vH) = 2^{\ex(n, H)}.$$
    \end{conjecture}

\subsection{Notation}
We denote by $[n]$ the set $\{1, \hdots, n\}$. By a graph we mean a pair $G = (V, E)$ where $V$ is a set of vertices and $E$ is a set of unordered pairs of vertices, called edges, i.e., $E\subseteq \binom{V}{2} = \{Y\subseteq V:|Y| = 2\}$. We write $xy$ for the edge $\{x, y\}$. All graphs do not contain loops or multiple edges. We refer to the set of vertices of a graph $G$ as $V(G)$ and the edges as $E(G)$, and denote $e(G) = |E(G)|$. A graph $H$ is said to be a subgraph of $G$, denoted $H\subseteq G$ if $V(H)\subseteq V(G)$ and $E(H)\subseteq E(G)$. If $H$ is a subgraph of $G$, we denote by $G\setminus H$ the graph formed by deleting from $G$ the vertices of $H$ and every edge incident to a vertex of $H$. If $F\subseteq E(G)$, $G\setminus F$ denotes the graph formed by deleting the edges in $F$ (but no vertices). Further, if $V'\subseteq V(G)$, then $G[V']$ is the graph formed by deleting all vertices not in $V'$ and all edges incident to a vertex not in $V'$, i.e., $G[V'] = (V', E\cap \binom{V'}{2})$. If $A, B\subseteq V(G)$ then $E(A, B)$ denotes the set of edges $xy\in E(G)$ with $x\in A$, $y\in B$. When ambiguity will not arise, we denote $E(G[A])$ simply by $E(A)$. For $v\in V(G)$, we denote by $N(v)$ the set of vertices in $G$ adjacent to $v$ and say the \textit{degree} of $v$ is $d(v) = |N(v)|$. If $A\subseteq V(G)$, $N_A(v)$ denotes $N(v)\cap A$ and  and $d_A(v) = |N_A(v)|$. 

A $p$\emph{-partite} graph $G$ is a graph such that there is a partition $V(G) = V_1\cup \hdots\cup V_p$ with $E(V_i) = \emptyset$ for each $i\in [p]$. We denote the complete $p$-partite graph where each part has size $k$ by $K_p(k)$.  The \textit{chromatic number} of a graph $G$, denoted $\chi(G)$, is the smallest number $p$ so that $G$ is $p$-partite. For a graph $G$ we call a $p$-partition $V_1\cup\hdots\cup V_p$ of $V(G)$ \textit{optimal} if it minimizes the number of edges within parts over all $p$-partitions. That is a partition which minimizes $\sum_{i=1}^p |E(V_i)|$. We denote by $T(n, p)$ the Tur\'{a}n graph with $p$ parts on $n$ vertices and denote $t(n, p) = |E(T(n, p))|$.

By a \textit{directed graph} we mean a pair $\vH = (V, E)$ where $V$ is a set of vertices, and $E\subseteq V\times V$ is a set of ordered pairs of vertices. Unless otherwise specified, we do not allow loops, multiple edges, or bidirected edges, i.e., if $(x, y)\in E$ then $(y, x)\notin E$ and for all $x$, $(x,x)\notin E$. For a directed graph $\vH$, we denote the vertex set of $\vH$ by $V(\vH)$ and the edge set of $\vH$ by $E(\vH)$.  By the \emph{underlying graph} of a directed graph $\vH$, we mean the graph obtained by forgetting the orientations placed on the edges of $\vH$. That is, if $\vH = (V, E)$, then the underlying graph of $\vH$ is $H = (V, E')$ where $E' = \{\{x, y\}: (x, y)\in E\}$.  Let $\vH$ be a directed graph and let $v$ be a vertex of $\vH$. We say $w$ is an \textit{in-neighbor} of $v$ if $(v, w)$ is an edge of $\vH$, and $w$ is an \textit{out-neighbor} of $v$ if $(w, v)$ is an edge of $\vH$. If $A, B\subseteq V(\vH)$ we denote $\vE(A, B) = \{(v, w)\in E(\vH): v\in A, w\in B\}$. In addition, we denote $E(A, B) = \vE(A, B)\cup \vE(B, A)$.  For any undefined notation regarding directed graphs we apply the notation to to the underlying graph $H$ of $\vH$. For example, we define the chromatic number $\chi(\vH)$ to be $ \chi(H)$ where $H$ is the underlying graph of $\vH$.

The following notion will appear repeatedly throughout the paper:
\begin{definition}
    Let $\vH$ be a directed graph, and let $U, V\subseteq V(\vH)$. If every edge between $U$ and $V$ is oriented towards $U$ or every edge between $U$ and $V$ is oriented towards $V$, we call the pair $(U, V)$ \emph{pure}.  Similarly, if there is a constant $t$ such that all but at most $t$ edges between $U$ and $V$ are oriented towards $U$ or all but at most $t$ edges between $U$ and $V$ are oriented towards $V$ then we call the pair $(U, V)$ $t$-\emph{almost pure}.
\end{definition} 
If $U$ consists of a single vertex $u$, we denote the pair $(\{u\},V)$ by $(u,V)$.

\section{Preliminary Lemmas}

In this section we present three lemmas that will be used repeatedly throughout the rest of the paper. The first lemma is a basic application of the directed regularity lemma, and is almost identical to Lemma 2.8 from \cite{bucic2023}. However, as the proof is slightly more general, we have included it in Appendix \ref{Proof_of_Lemma}.

\begin{lemma}
\label{main_lemma}
    Let $\vH$ be a directed graph with chromatic number $\chi(\vH)\geq 3$, and let $p = \chi(\vH) -1$. For any $\delta>0$, there exists $n_0$ such that for any graph $G$ with $V(G) = n\geq n_0$ and at least $2^{t(n,p)}$ many $\vH$-free orientations, there is a partition $V(G) = V_1\cup\hdots \cup V_p$ such that $|E(V_1)| +\hdots + |E(V_p)|\leq \delta n^2$.
\end{lemma}

Let $G$ be a graph with $n$ vertices, and suppose for some $\delta >0$, we can partition $V(G)$ into parts $V_1\cup \hdots \cup V_p$ such that $|E(V_1)|+ \hdots +|E(V_p)|<\delta^2n^2$. As in \cite{bucic2023}, given a partition $V(G) = V_1\cup \hdots \cup V_p$ we call an orientation $\vG$ of $G$ \textit{relevant} if, for any $X_1\subseteq V_i$, $X_2\subseteq V_j$ with $i\neq j$ and $|X_1|, |X_2|>2\delta n$, there are at least $|X_1||X_2|/10$ edges directed from $X_1$ to $X_2$. The following lemma shows that if a graph $G$ looks approximately like a $p$-partite graph, then $G$ does not have many non-relevant orientations (and thus if $G$ has many edges, then $G$ has many relevant orientations). 
\begin{lemma}
    Let $\delta>0$ and let $n$ be sufficiently large. Let $G$ be a graph on $n$ vertices and let $V(G) = V_1\cup\hdots\cup V_p$ be a $p$-partition of $V(G)$ such that 
    $$|E(V_1)|+ \hdots +|E(V_p)|<\delta^2n^2.$$ 
    Then there are at most $2^{t(n, p)-1}$ non-relevant orientations of $G$.
    \label{relevant_orientations_lem}
\end{lemma}
\begin{proof}
Fix $X_1\subseteq V_i, X_2\subseteq V_j$ with $i\neq j$ and $|X_1|, |X_2|\geq 2\delta n$. There are
$$\sum_{i = 0}^{|X_1||X_2|/10}{\binom{|E(X_1, X_2)|}{ i}}\leq \sum_{i = 0}^{|X_1||X_2|/10}{\binom{|X_1||X_2|}{ i}}\leq2^{|X_1||X_2|/2}$$
many ways to orient the edges between $X_1$ and $X_2$ so that there are at most $|X_1||X_2|/10$ edges directed from $X_1$ to $X_2$. Since $G$ has at most $t(n, p)+\delta^2n^2$ edges, there are at most 
$$2^{t(n, p)+\delta^2n^2-|X_1||X_2|}2^{|X_1||X_2|/2} = 2^{t(n, p)+\delta^2n^2-|X_1||X_2|/2}\leq 2^{t(n,p)-\delta^2n^2}$$ 
ways to orient the whole graph such that there are at most $|X_1||X_2|/10$ edges from $X_1$ to $X_2$.  Finally, there are at most $2^{2n}$ ways to choose $X_1, X_2$. So, for large enough $n$, there are at most
$$2^{2n}2^{t(n, p)-\delta^2n^2}\leq 2^{t(n, p)-1}$$ of these non-relevant orientations.
\end{proof}

Next we have a lemma that we will use repeatedly.

\begin{lemma}
\label{embedding_lem_dir}
    Let $0<\delta<\frac{1}{8k}$ and suppose $\vH$ is a directed graph with $V(\vH) = k$ and $\chi(\vH) \leq p$. Then there exists $n_0$, depending only on $\vH$, such that for any $n>n_0$ the following holds: Let $G$ be a $p$-partite graph on $n$ vertices with parts $V_1, \hdots, V_p$ each of size at least $10^k\delta n$ and assume $\vG$ is a relevant orientation of $G$. Let $c: V(\vH)\to [p]$ be a proper coloring of $\vH$. Then there is an embedding of $\vH$ into $\vG$ such that a vertex $v$ of $\vH$ lies in $V_j$ whenever $c(v) = j$.
\end{lemma}
\begin{proof}
    We prove this by induction on $k$. Clearly the statement holds for all $p$ when $k = 1$. 
    
    Suppose now that  $k>1$ and the statement holds for $k-1$. Let $\vH$ be a directed graph on $k$ vertices with $\chi(\vH)\leq p$. Let $c:V(\vH)\to [p]$ be a proper coloring of $\vH$. Fix a labeling $V(\vH) =\{v_1, \hdots, v_k\}$. 
    
    Fix a pair $(i, j)$ with $i\neq j \in [k]$. Define $V_{i,j}^1$ to be those vertices $v\in V_i$ with fewer than $\frac{1}{10}|V_j|$ in-neighbors in $V_j$ and let $V_{i,j}^0$ be those vertices $v\in V_i$ with fewer than $\frac{1}{10}|V_j|$ out-neighbors in $V_j$. Then we have $|\vE(V_{i, j}^1, V_j)|\leq \frac{1}{10}|V_{i, j}^1||V_j|$, so since $\vG$ is relevant, $V_{i, j}^1\leq 2\delta n$. Similarly, $V_{i, j}^0\leq 2\delta n$. Thus, we have 
    $$|\bigcup_{j\neq i}(V_{i, j}^1\cup V_{i, j}^0)|\leq 4k\delta n\leq \frac{1}{2}|V_i|.$$ Therefore, there is a vertex $v\in V_i$ such that $v$ has at least $\frac{1}{10}|V_j|$ in-neighbors and $\frac{1}{10}|V_j|$ out-neighbors in $V_j$.  Let 
    $$V_j^1 = \{x\in V_j: (x, v)\in E(\vG)\} \text{ and } V_j^0 = \{x\in V_j: (v, x)\in E(\vG)\}.$$
    Then $|V_j^t|\geq \frac{1}{10}|V_j|\geq 10^{k-1}\delta n$ for $t\in \{0, 1\}$.  
    
    Define a new coloring $c': V(\vH)\setminus \{v\}\to [p]\times \{0,1\}$ as follows: For each $x\in V(\vH)$ with $c(x)=j$, set $c'(x) = (j, 1)$ if $(x, v)$ is an edge of $\vH$, and set $c'(x) = (j, 0)$ otherwise.  Note $c'$ is a proper coloring of $\vH\setminus\{v\}$. Consequently, by induction, there are vertices $\{w_1,\hdots, w_{k-1}\}\subseteq (V_i\setminus \{v\})\cup\bigcup_{j\neq i}(V_j^1\cup V_j^0)$ such that 
    $$\vH\setminus \{v_k\} \subseteq \vG[\{w_1,\hdots, w_{k-1}\}] $$ 
    and for each $s\in [k]$, $w_s$ lies in $V_j^t$ if $c'(v_s) = (j,t)$ and $j \neq i$ and $w_s$ lies $V_i\setminus\{v\}$ if $c(w_s) = i$. Then by construction, 
    $$\vH \subseteq G[\{v\}\cup \{w_1,\hdots,w_k-1\}].$$
\end{proof}

\section{Proof of Theorems \ref{general_graph_theorem} and \ref{general_gph_finite_thm}}
In this section we prove Theorems \ref{general_graph_theorem} and \ref{general_gph_finite_thm}.  We begin with a lemma, Lemma \ref{general_thm_lem}, which will form the bulk of the proof of both theorems.  The idea of Lemma \ref{general_thm_lem} is as follows: For any constant $a$, if $G$ has $2^{t(n, p) +m}$ many $\vH$-free orientations for some $m>0$, then there is some number $q$ and vertices $v_1, \hdots, v_q$ can be deleted with at least $2^{t(n-q, p) +m +a}$ many $\vH$-free orientations remaining. Though the number of $\vH$-free orientations of $G$ may not increase after deleting these $q$ vertices, we will have
$$\frac{\log(D(G, \vH))}{t(n, p)}< \frac{\log(D(G \setminus \{v_1, \hdots, v_q\}, \vH))}{t(n-q, p)},$$ 
which turns out to be sufficient for our purposes.
\begin{lemma}
\label{general_thm_lem}
    Let $\vH$ be a directed graph satisfying the conditions of Theorem \ref{general_graph_theorem}. Let $a>0$. Suppose $n$ is sufficiently large, $G$ is a graph on $n$ vertices, and $M\in \cM'(\vH)$ is such that the following hold:
    \begin{itemize}
        \item $G$ has $2^{t(n, p) + m}$ many $\vH$-free orientations for some $m>0$.
        \item There is an optimal $p$-partition $V(G) = V_1\cup \hdots \cup V_p$ of $G$ such that $M$ is a subgraph of $V_1$.
    \end{itemize}
    Then there is some number $q\in \{1, |V(M)|\}$ and vertices $v_1, \hdots, v_q$ such that $G\setminus \{v_1, \hdots, v_q\}$ has at least
    $2^{t(n-q, p) +m + a}$ many 
    $\vH$-free orientations.
\end{lemma}
By close inspection of the proof of Lemma \ref{general_thm_lem} we note that $a$ may be replaced with any sublinear function. Before proceeding with proof, we prove Theorems \ref{general_graph_theorem} and \ref{general_gph_finite_thm} assuming Lemma \ref{general_thm_lem}. The technique used in the following proofs is similar to that used to prove Theorem 1.1 in \cite{bucic2023} and Theorem 1.1 in \cite{Alon2006}.

\begin{proof}[Proof of Theorem \ref{general_graph_theorem}]
As noted before, the lower bound is immediate. Let $n_0$ be the value $n$ obtained from Lemma \ref{general_thm_lem} and let $G$ be a graph on $n>n_0^2 + n_0$ vertices such that $D(G, \vH) = D(n, \vH)$.  Suppose for contradiction that $G$ has  
$$2^{t(n, p) + p\ex(n, \cM'(\vH)) + m}$$
many $\vH$-free orientations for some $m>0$. Let $V_1, \hdots, V_p$ be an optimal $p$-partition of $G$. Then since, in particular, $|E(G)|\geq t(n, p) + p\ex(n, \cM'(\vH)) + m$, there is a part (say without loss of generality $V_1$) and $M\in \cM'(\vH)$ such that $M\subseteq G[V_1]$. Let $t = \max\{|V(L)|:L\in \cM'(\vH)\}$. By Lemma \ref{general_thm_lem}, for some $q>0$, there are $q$ vertices $v_1, \hdots, v_q$ such that $G\setminus\{v_1, \hdots, v_q\}$ has at least 
$$2^{t(n-q, p) + p\ex(n, \cM'(\vH)) +m+t}$$ 
many $\vH$-free orientations. Further note that this implies $G\setminus\{v_1, \hdots, v_q\}$ has at least $t(n-q, p) + p\ex(n, \cM'(\vH)) +m+t$ many edges, and therefore we may apply Lemma \ref{general_thm_lem} again, this time to $G\setminus\{v_1, \hdots, v_q\}$. Since $G$ has $n>n_0^2 + n_0$ many vertices, we may repeat this deletion process until we obtain a graph $G'$ with $l\geq n_0$ vertices and at least $2^{t(l, p) + p\ex(n, \cM'(H)) + l^2}\geq 2^{l^2}$ many $\vH$-free orientations. This implies that $G'$ has at least $l^2$ edges, a contradiction.
 
\end{proof}

Next we prove Theorem \ref{general_gph_finite_thm}.

\begin{proof}[Proof of Theorem \ref{general_gph_finite_thm}]
    Let $n_0$ be the value $n$ obtained from Lemma $\ref{general_thm_lem}$. Let $G$ be a graph on $n>n_0^2 + n_0$ vertices which maximizes $D(G, \vH)$. Suppose $\ex(n, \cM'(\vH))$ is a constant $a$ for large enough $n$. Note that $G$ has at least $2^{t(n, p)}$ many $\vH$-free orientations, and therefore at least $t(n, p)$ edges. Let $V(G) = V_1\cup \hdots\cup V_p$ be an optimal $p$-partition of $G$. Suppose for contradiction that $\sum_{i\in [p]}|E(V_i)|>pa$. Then there is $M\in \cM(\vH)$ such that $M\subseteq G[V_i]$ for some $i$. Thus we may apply Lemma \ref{general_thm_lem} with $a = p\ex(n, \cM'(\vH)) + 2$ to obtain a graph $G\setminus\{v_1, \hdots, v_q\}$ for some $q\leq \max\{|V(M)|: M\in \cM'(\vH)\}$ with at least $2^{t(n-q, p) + p\ex(n, \cM'(\vH))+2}$ many $\vH$-free orientations. Therefore, we may proceed as in the proof of Theorem \ref{general_graph_theorem}, repeatedly applying Lemma \ref{general_thm_lem} until we obtain a graph with $l\geq n_0$ vertices and at least $l^2$ edges, a contradiction.
\end{proof}

We may now proceed to the proof of Lemma \ref{general_thm_lem}. 

\begin{proof}[Proof of Lemma \ref{general_thm_lem}]
    Suppose $\vH$ has $k$ vertices.  Fix $\delta>0$ sufficiently small and let $n_0$ be the value obtained by applying Lemma \ref{main_lemma} with the parameter $\delta^2$ in place of $\delta$.  Let $G$ be a graph with $n\geq n_0$ vertices and suppose $G$ has exactly $2^{t(n,p)+m}$ many $\vH$-free orientations for some $m>0$.  Let $V(G) = V_1\cup \hdots \cup V_p$ be an optimal $p$-partition of $G$.  Note that by Lemma \ref{main_lemma}, $|E(V_1)| +\hdots + |E(V_p)|\leq \delta^2n^2$, and consequently we have at most $\delta^2n^2$ edges missing between $V_i$ and $V_j$ for any $i\neq j\in [p]$.  In addition, it is not difficult to show that 
    $$n(\frac{1}{p}-\delta)\leq |V_1|,\hdots, |V_p|\leq n(\frac{1}{p}+\delta).$$

By Lemma \ref{relevant_orientations_lem}, there are at most $2^{t(n, p)-1}$ many non-relevant orientations of $G$, and therefore there are at least $2^{t(n,p)+m-1}$ many relevant, $\vH$-free orientations of $G$. We now analyze three cases.

\begin{case} Assume there is a vertex $v\in V(G)$ such that $v$ has at least $10^k\delta n$ neighbors in $V_i$, i.e. $|N_G(v)\cap V_i|\geq 10^k\delta n$, for each $i\in [p]$.
\end{case}
We will show that there is some $c<1-\frac{1}{p}$ such that $G\setminus\{v\}$ has at least 
$$2^{t(n,p)+m - 1- cn} \geq 2^{t(n-1,p)+m+a}$$ 
many $\vH$-free orientations, where the inequality holds since $n$ is large compared to $a$. Fix a relevant orientation $\vG$ of $G$. We claim that if $\vG$ is $\vH$-free, then there are at least two parts $V_i, V_j$ such that $(v, V_i)$ and $(v, V_j)$ are $10^k\delta n$-almost pure. Suppose towards a contradiction that there is a unique part, $V_j$, such that for all $i\neq j$, $v$ has at least $10^k\delta n$ in-neighbors and $10^k \delta n$ out-neighbors in $V_i$. Suppose without loss of generality that $j = 1$. Further, without loss of generality, suppose that $v$ has more out-neighbors in $V_1$ than in-neighbors in $V_1$. For each $i \in [p]$, define $N_i = N_{V_i}(v)$. Let $N_{i}^{0} = \{x\in N_{i}: (x, v)\in E(\vG)\}$ and $N_{i}^{1} = \{x\in N_{i}:(v,x)\in E(\vG)\}$ i.e. $N_i^0$ is the out-neighbors of $v$ in $V_i$ and $N_i^1$ is the in-neighbors of $v$ in $V_i$. Then by assumption $|N_i^{0}|\geq 10^k\delta n$ for all $i\geq 1$ and $|N_i^{1}|\geq10^k\delta n$ for all $i\geq 2$. 

Recall that $\cM(\vH)$ contains both pure orientations of a $(t+1)$-star, $S_{t+1}$. Note that any proper 2-coloring of $S_{t+1}$ has one color for the center vertex and every other vertex is the other color. Thus, there is a vertex $w$ of $\vH$ such that $\vH\setminus\{w\}$ is $p$-colorable. Let $c:V(\vH)\to [p+1]$ be a coloring of $\vH$ such that the following holds: $w$ is the unique vertex of color $p+1$, and $w$ has exactly $t$ neighbors of color $1$ which, along with $w$, form a pure $(t+1)$-star. Since both pure orientations of $S_{t+1}$ are in $\cM(\vH)$, we may choose this $S_{t+1}$ to be such that $w$ has only out-neighbors in $\vH[S_{t+1}]$. Note that $c|_{V(\vH)\setminus\{w\}}$ only uses colors $1$ through $p$. We define a new coloring $c':V(\vH)\setminus\{w\}\to [p]\times \{0,1\}$ as follows: for $x\in V(\vH)\setminus\{w\}$, if $x$ is an out-neighbor of $w $, then $c'(x) = (c(x), 0)$, otherwise, $c'(x) = (c(x), 1)$. Note this is a proper coloring of $\vH$. Also, if $x,y\in V(\vH)\setminus\{w\}$ are such that $c'(x) = (i, 0)$ and $c'(y) = (i, 1)$, then there is no edge between $x$ and $y$ in $\vH$.  Thus, by Lemma \ref{embedding_lem_dir}, since our fixed orientation $\vG$ of $G$ is relevant, we may find a copy of $\vH\setminus\{w\}$ in $N(v)$ such that if $c'(x) = (i, 0)$, then $x$ lies in $N_i^{0}$ and if $c'(x) = (i, 1)$, then $x$ lies in $N_i^{1}$. By construction, this subgraph can be extended to a copy of $\vH$ by adding the vertex $v$ as the center of $S_{t+1}$. Therefore, we see that any relevant, $\vH$-free orientation of $G$ must have the property that there are at least two parts, $V_i, V_j$ such that $(v, V_i)$ and $(v, V_j)$ are both $10^k\delta n$-almost pure. Therefore, we have the following:
\begin{itemize}
    \item There are at most $2^{(p-2 +\delta)\frac{n}{p}}$ many ways to orient the edges between $v$ and $V_k$ $k\notin \{i, j\}$.
    \item There are at most 
    $$2\sum_{i=1}^{10^k\delta n}\binom{n}{i}$$
    ways to orient the edges between $v$ and each of $V_i$ and $V_j$.
    \item There are $\binom{p}{2}$ ways to choose $i, j$.
\end{itemize}
Thus, putting this together, the number of ways to orient the edges incident to $v$ is at most
$$2^{(p-2 +\delta)\frac{n}{p}}4\binom{p}{2}\left(\sum_{i=1}^{10^k\delta n}\binom{n}{i}\right)^2\leq 2^{cn}.$$
 For some $c<(1-\frac{1}{p})$.  Consequently, the number of $\vH$-free orientations of $G\setminus\{v\}$ is at least 
$$2^{t(n,p)+m - 1- cn} \geq 2^{t(n-1,p)+m+a}$$ 
as desired.
\begin{case}
    Assume we are not in Case 1 and there is $v\in V_i$ such that $v$ has fewer than $n(1-1/p)-2(10^k)\delta n$ neighbors in $\bigcup_{j\neq i}V_j$.
\end{case}   
Since we are not in Case 1 and since our partition $V_1\cup \hdots\cup V_p$ is optimal, we may assume that every vertex in $V(G)$ has at most $10^k\delta n$ neighbors in its own part.  Then 
$$d(v)\leq 10^k\delta n +n(1-\frac{1}{p})-2(10^k)\delta n\leq n(1-\frac{1}{p})-0.5(10^k)\delta n.$$
Therefore, the number of ways to orient the edges incident to $v$ is at most $2^{n(1-\frac{1}{p})-0.5(10^k)\delta n} \leq 2^{cn}$ for some $c<1-\frac{1}{p}$.  Thus the number of $\vH$-free ways to orient $G\setminus\{v\}$ is at least
$$2^{t(n,p)+m-1 - cn} \geq 2^{t(n-1,p)+m+a}$$
as desired.
\begin{case}
    Assume we are not in Cases 1 or 2.
\end{case} 
By assumption, there is a copy of some $M\in \cM'(\vH)$ contained in $G[V_i]$ for some $i\in[p]$, without loss of generality say $G[V_1]$. Suppose $M$ has $r$ vertices $v_1, \hdots, v_r$. We will show that for some $c<(1-\frac{1}{p})$, $G\setminus M$ has at least 
$$2^{t(n,p)+m - 1- crn} \geq 2^{t(n-r,p)+m+a}$$ 
many $\vH$-free orientations. 

Since $M\in\cM'(\vH)$, for any orientation $\vM$ of $M$, there are vertices $w_1, \hdots, w_r\in V(\vH)$ such that $\vH[\{w_1, \hdots, w_r\}]\cong \vM$ and $\vH\setminus \{w_1, \hdots, w_r\}$ is $p$-colorable. Fix a relevant orientation $\vG$ of $G$. Let $\vM$ be the orientation of $M$ found as a subgraph of $\vG$.  Let $c: V(\vH)\to [p]$ be a proper coloring of $H$ such that $c(V(\vH\setminus \vM))\subseteq [p]$ and $c(V(\vM)) = \{1, p+1\}$.  Let $A$ be the neighborhood of $\{v_1, \hdots, v_r\}$ in $G$; that is, $A = \bigcap_{i=1}^rN(v_i)$.  Then because we are not in Case 2, $|A|\geq n((1-\frac{1}{p})-2(10^k)r\delta)$. In particular, $|A\cap V_i| \geq \frac{n}{p}-2(10^k)r\delta n$ for each $i\neq1$. 

Define $N_i = A\cap V_i$ for each $i\neq 1$.  We denote by $2^{[r]}$ the set 
$$\{\vec{x}=(x_1, \hdots, x_r): x_i\in\{0,1\}, i\in [r]\}.$$ For each vector $\vec{x} = (x_1, x_2, \hdots, x_r)\in 2^{[r]}$, and each $i$, define $N_i^{\vec{x}}$ to be the set of vertices $u\in N_i$ such that $(u,v_j)\in E(\vG)$ if and only if $x_j = 0$ (and so if $x_j = 1$, then either there is no edge between $u$ and $v_j$ or $(v_j, u)\in E(\vG)$). We claim that there is at least one $i\in \{2,\ldots, p\}$ and one $\vec{x}\in 2^{[r]}$ such that $|N_i^{\vec{x}}|\leq 10^k\delta n$. Suppose for contradiction that for each $i\in [p]$ and each $\vec{x}\in 2^{[r]}$ that $|N_i^{\vec{x}}|\geq 10^k\delta n$. Then we define a new coloring $c':V(\vH\setminus\vM)\to[p]\times 2^{[r]}$ by $c'(u) = (c(u), \vec{x})$  where $\vec{x}$ is such that $x_j = 0$ whenever $(u, v_j)$ is an edge of $\vH$ and $x_j = 1$ otherwise. Then by a similar strategy to that used in Case 1, we note that this is a proper coloring of $\vH\setminus\vM$, and if $c'(u) = (i, \vec{x})$ and $c'(v) = (i, \vec{y})$ then there is no edge between $u$ and $v$. By Lemma \ref{embedding_lem_dir}, we can find a copy of $\vH\setminus\vM$ in $\vG\setminus \vM$ such that if $c'(v) = (i, \vec{x})$, then $v$ lies in $N_i^{\vec{x}}$. By construction, this can be extended to a copy of $\vH$ by adding in the vertices $v_1, \hdots, v_r$. Thus, there must be at least one $i\in \{2, \hdots, p\}$ and one $\vec{x}\in 2^{[r]}$ such that $|N_i^{\vec{x}}|\leq 10^k\delta n$. 

We count the number of ways to orient the edges incident to some vertex in $\vM$. Say we've chosen $N_i^{\vec{x}}$ to have size at most $10^k\delta n$. Then:
\begin{itemize}
    \item There are at most $2^{10^kr\delta n}$ many ways to orient the edges adjacent to $M$ that live entirely within $V_1$.
    \item There are at most $2^{rn(1-\frac{2}{p})}$ many ways to orient the edges between $M$ and $\bigcup_{j\in [p]\setminus \{1, i\}}V_j$.
    \item For some $0<b<1$, there are at most 
$$\sum_{j=0}^{10^k\delta n}\binom{\frac{n}{p}}{j}(2^r-1)^{\frac{n}{p}-j}\leq 10^k\delta n \binom{\frac{n}{p}}{10^k\delta n}(2^r-1)^{\frac{n}{p}-10^k\delta n}< 2^{\frac{brn}{p}}$$
many ways to orient the edges between $M$ and $V_i$.
\end{itemize}
 Further, there are $2^rp$ ways to choose $N_i^{\vec{x}}$. Thus, the total number of ways to orient the edges incident to a vertex in $M$ is at most
$$2^{10^kr\delta n}2^{rn(1-\frac{2}{p})}2^rp2^{\frac{brn}{p}} < 2^{crn}$$
for some constant $c<(1-\frac{1}{p})$.  Therefore, the number of $\vH$-free ways to orient the graph $G\setminus M$ is at least 
$$2^{t(n,p)+m - 1- crn} >2^{t(n-r,p)+m+a}$$ 
as desired.
\end{proof}

\section{Results for $(k, r)$-Fans}
In this section, we prove Corollary \ref{kfans_cor} and Theorems \ref{k-fans_thm_1} and \ref{k-fans_thm_2}. Recall that a $(k, r)$-fan, denoted $F_{k,r}$ is $k$ copies of $K_r$ all joined at a single vertex. We call the vertex common to all copies of $K_r$ in $F_{k, r}$ the \textit{center} of $F_{k,r}$.  If $K$ is a specific copy of $K_r$ in $F_{k,r}$, it will be useful to refer to the center of $K$, meaning the distinguished vertex of $K$ which, when considered as part of $F_{k, r}$ is the center of $F_{k,r}$. By \cite{CHEN2003159}, for all $k\geq 1$ and $r\geq 2$, $\cM(F_{k, r})$ is comprised of a $(k+1)$-star and a $k$-matching. We prove Corollary \ref{kfans_cor}.

\begin{proof}[Proof of Corollary \ref{kfans_cor}]
   Let $v$ be the center vertex of $F_{k, r}$. Let $\vF_{k, r}$ be an anti-directed orientation of $F_{k,r}$, i.e., one such that each copy of $K_r$ has at least one edge entering $v$ and one edge exiting $v$. One can check that we have chosen our orientation of $F_{k, r}$ such that $\cM(\vF_{k, r}) = \overrightarrow{\cM(F_{k, r})}$.  Thus we may apply Theorem \ref{general_graph_theorem} to obtain Corollary \ref{kfans_cor} immediately. 
\end{proof}

Recall that we have denoted
$$h(k) = \begin{cases}
    k^2-k & \text{ if $k$ is odd}\\
    k^2 -\frac{3}{2}k & \text{ if $k$ is even}.
    \end{cases}$$
and by \cite{CHEN2003159}, we have $\ex(n, F_{k, r}) = t(n, r-1) +h(k)$.

Before we proceed to the proof of Theorems \ref{k-fans_thm_1} and \ref{k-fans_thm_2}, we prove a general lemma, which will be crucial in the following proofs.

\begin{lemma}
\label{parts_lem}
    Let $\vH$ be a directed graph satisfying the assumptions of Theorem \ref{general_gph_finite_thm} with $\chi(\vH) = p+1$. Let $G$ be a graph maximizing $D(G, \vH)$ and $V(G) = V_1\cup \hdots \cup V_p$ an optimal $p$-partition of $G$.  Then for all $i\in [p]$ and all $M\in \cM'(\vH)$, $M\not\subseteq G[V_i]$.
\end{lemma}
\begin{proof}
    Suppose $\vH$ has $k$ vertices. Let $G$ be a graph on $n$ vertices such that  $D(G, \vH) = D(n, \vH)$, and let $V(G) = V_1\cup\hdots\cup V_p$ be an optimal $p$-partition of $G$. Note that since $\ex(n, H) \geq t(n, p)$, $G$ has at least $t(n, p)$ edges, and by Theorem \ref{general_gph_finite_thm}, $\sum_{i\in [p]}|E(G[V_i])|\leq pa$. Therefore, there are at most $pa$ edges missing from between the parts $V_1, \hdots, V_p$.  Further it is easy to show that $|V_i| = \frac{n}{p}\pm pa$ for all $i\in [p]$.
    
    Suppose for contradiction that there is $M\in \cM'(\vH)$ such that $M\subseteq G[V_1]$. For each $i\in [p]$, let $A_i$ be the set of vertices $v\in V_i$ such that for all $w\in V_j$ with $j\neq i$, we have $vw\in E(G)$. Label the vertices $V(M)\subseteq V_1$ by $v_1, \hdots, v_q$.  For each $i\in \{2, \hdots, p\}$, let 
    $$N_i = A_i\cap \bigcap_{j=1}^qN(v_j)$$
    and let $N_1 = A_1\setminus\{v_1, \hdots, v_q\}$. Then since there are at most $pa$ edges missing between parts, $|N_i|\geq \frac{n}{p}-4pa$ for each $i\in [p]$. We partition each set $N_i$ into parts $N_i = S_i^0\cup S_i^1\cup \hdots\cup S_i^t$ so that $S_i^j$ has size $k$  for $j\geq 1$ and $|S_i^0|<k$. Note $t\in \{\lfloor\frac{n}{pk}\rfloor, \lceil\frac{n}{pk}\rceil\}$. Then for each $j\in [t]$, $G[\bigcup_{i=1}^pS_i^j]\cong K_p(k)$.  Since $H\setminus M$ is $p$-colorable, $H\setminus M$ is contained as a subgraph in $\bigcup_{i=1}^pS_i^j$ for all $j\in [t]$. Further, by construction, each of these copies of $H\setminus M$ can be extended to $H$ by adding the vertices $v_1, \hdots, v_q$.  

    Choose an orientation $\vG$ of $G$ uniformly at random. Let $\vM$ be the orientation of $M$ in $\vG$. Then since $M\in \cM'(\vH)$, regardless of the orientation $\vM$, there is a way to extend $\vM$ to $\vH$. In particular for each $j\in t$, the probability that $\vG[\{v_1, \hdots, v_q\}\cup \bigcup_{i=1}^{p}S_{i}^j]$ does not contain a copy of $\vH$ is at most 
    $$1-\frac{1}{2^{|E(\vH)|-|E(\vM)|}}.$$
    Therefore, the probability that $\vG$ does not contain a copy of $\vH$ is at most 
    $$\left(1-\frac{1}{2^{|E(\vH)|-|E(\vM)|}}\right)^t< \frac{1}{2^{pa+1}}$$
    for large enough $n$. Therefore the number of $\vH$-free orientations of $G$ is at most
    $$\frac{1}{2^{pa+1}}2^{|E(G)|}\leq \frac{1}{2^{pa+1}}2^{t(n, p) + pa} <2^{t(n,p)}$$
    a contradiction.
\end{proof}

Next, we prove Theorem \ref{k-fans_thm_1}.

\begin{proof}[Proof of Theorem \ref{k-fans_thm_1}]
 Let $n$ be the value obtained from Corollary \ref{kfans_cor}, and let $G$ be a graph on $n$ vertices such that $D(G, \vF_{2,r}) = D(n, \vF_{2,r})$. Since $\ex(n,F_{2,r}) = t(n, r-1) + 1$, we have $D(G, \vF_{2,r})\geq 2^{t(n, r-1) + 1}$. Suppose for contradiction that $D(G, \vF_{2,r})> 2^{t(n, r-1) + 1}$. Note this implies that $|E(G)|>t(n, r-1)+1$. Recall that $\cM'(\vF_{2, r})$ consists of a 2-matching and a 3-star. Thus 
 $$\ex(n, \cM'(\vF_{2, 3})) = 1.$$
 Therefore, by Corollary \ref{kfans_cor} we have 
    $$D(n, \vF_{2, r})\leq2^{t(n, r-1) + (r-1)},$$
    and in particular, by Lemma \ref{parts_lem}, there is a partition $V(G) = V_1\cup\hdots\cup V_{r-1}$ such that there is at most one edge entirely contained in $V_i$ for each $i\in [r-1]$. Suppose that there is exactly one edge in each of $V_1, \hdots, V_\ell$ and no edges in $V_{\ell + 1},\hdots, V_{r-1}$. Call these edges $w_1x_1, \hdots, w_{\ell}x_{\ell}$.  We will show that $G$ has fewer than $2^{t(n, r-1) + 1}$ many $\vF_{2, r}$-free orientations, a contradiction.  
    
    We now make a few observations about the structure of $G$.  First note that $G$ has more than $t(n, r-1) + 1$ edges only if there are at most $\ell-2$ edges missing in total from between any parts $V_i$ and $V_j$. Indeed, suppose $G$ has exactly $t(n, r-1) + s$ edges.  Then $2\leq s\leq \ell $. Further, for each $i\in [r-1]$, we have $|V_i|\geq\frac{n}{r-1}-\ell$ since otherwise $G$ would have fewer than $t(n, r-1) + 1$ edges. For each $i\in [r-1]$, let $A_i$ be the set of vertices $v\in V_i$ such that $v$ has an edge to every vertex $u\in V_j$ for $j\neq i$ and and no edges to other vertices in $V_i$. In other words
    $$A_i = \{v\in V_i: \forall u\in V(G)\setminus V_i, \,uv\in E(G)\}\setminus\{w_i, x_i\}$$
    Note that since there are at most $\ell -2$ edges missing from between all parts $V_i$ and $V_j$ and since each such missing edge is ``incident" to only two vertices, we have  $|A_i| \geq |V_i|-\ell\geq \frac{n}{r-1}-2\ell$ for each $i$.  
    
    Since $G$ has $t(n, r-1)+s$ edges, it suffices to show that for an orientation $\vG$ of $G$ chosen uniformly at random, the probability that $\vG$ does not contain a copy of $\vF_{2, r}$ is less than $\frac{1}{2^{s-1}}$ as this implies that $G$ has fewer than
    $$\frac{1}{2^{s-1}}2^{|E(G)|}= \frac{1}{2^{s-1}}2^{t(n, r-1)+s} = 2^{t(n, r-1)+1}$$
    many $\vF_{2, r}$-free orientations.
    
    By definition, in our directed graph $\vF_{2, r}$, we have two oriented copies of $K_r$ as subgraphs.  Denote these by $\vK_r^1$ and $\vK_r^2$. We consider the quadruples of vertices $\{w_i, x_i, w_j, x_j\}$ for $i\neq j\in [\ell]$.  There are $\binom{\ell}{2}$ such quadruples and since there are at most $\ell-2$ edges missing from between the parts $V_1, \hdots, V_{r-1}$, at least $\binom{\ell}{2}-\ell+2$ of them form a copy of $K_4$ in $G$. One may easily check that $\binom{\ell}{2}-\ell+2\geq s-1$ since $s\leq \ell$. 
    
    We now restrict our attention to one such copy of $K_4$. Let $X \defeq \{w_1, x_1, w_2, x_2\}$ and without loss of generality say $G[X]\cong K_4$. Since our orientation $\vF_{2, r}$ is anti-directed, there is at least one edge entering our center vertex $v$ and at least one edge exiting $v$ in each copy of $K_r$. Therefore each of $\vK_r^1$ and $\vK_r^2$ contains a triangle of one of the following forms:
    \begin{itemize}
        \item $\{(v,y), (y, z), (z, v)\}$
        \item $\{(v,y), (z, y), (z, v)\}$
    \end{itemize} 
    for some vertices $y, z$ from $\vF_{2, r}$. 
    
    Fix such a triangle in $\vK_r^1$, and call it $\vT$. Let $\vG$ be an orientation of $G$ chosen uniformly at random. We wish to find a copy of $\vT$ in $\vG[X]$ that includes the edges $(u, w_i)$ and $(x_i, u)$ for some $i\in [2]$ where $u\in \{w_{3-i}, x_{3-i}\}$. By listing the 64 possibilities of $G[X]$ and counting we see the probability that we have such a copy of $\vT$ is $5/8$. Thus, the probability that a random orientation of $G[X]$ does not contain such a copy of $\vT$ is $3/8$.  
    
    Now, suppose $\vG[X]$ does contain such a copy of $\vT$. Suppose without loss of generality that the vertices of this $\vT$ are $w_1x_1w_2$. Let $t = \min_{i}|A_i|$. We partition 
    $$\bigcup_{i = 3}^{r-1}A_i = S_0\cup S_1\cup \hdots \cup S_t$$
    into disjoint copies of $K_{r-3}$ where each part $S_i$ consists of exactly one vertex from each part $A_3, \hdots, A_{r-1}$ for all $i\in[t]$ and let $S_0$ be the remaining vertices. Note that by the definition of $A_i$, $G[\bigcup_{i = 3}^{r-1}A_i]$ is a complete $(r-3)$-partite graph. Further, by our bound on $|A_i|$ we can guarantee we have $t = \min_{i}|A_i|\geq \frac{n}{r-1}-2\ell$. By construction, for each $i\in [t]$, we have $G_i \coloneqq G[S_i\cup \{w_1, x_1, w_2\}]\cong K_r$. Further, by our choice of $\vT$, there is at least one orientation $\vG_i$ of $G_i$ and a directed graph isomorphism $\phi:\vG_i\to \vK_r^1$ such that $\vG_i[\{w_1, x_1, w_2\}] \cong \vT$ and $\phi(w_2) = v$, where $v$ is the distinguished center vertex of $\vK_r^1$. 
    
    There are $\binom{r}{2}$ edges in $G_i$ for each $i\in[2]$ and so there are $2^{\binom{r}{2}}$ ways that each $G_i$ could be oriented. Having already assumed that $\vG[\{w_1, x_1, w_2\}] \cong \vT$ via $\phi$, we know that $\vG[\{w_1, x_1, w_2\}] = \vG_i[\{w_1, x_1, w_2\}]$ for each $i\in[t]$. Therefore, the probability that $\vG[G_i] = \vG_i$ is 
    $$\frac{1}{2^{\binom{r}{2}-3}}.$$
    Thus, the probability that  $\vG[G_i] \neq \vG_i$ for all $3\leq i\leq t$ is at most
    $$\left(1-\frac{1}{2^{\binom{r}{2}-3}}\right)^{\frac{n}{r-1}-2\ell},$$
    which tends to 0 as $n\to \infty$. 
    
    Now, suppose we have found $i$ such that $\vG[G_i]\cong \vK_r^1$ as described above. Suppose its vertices are $\{w_1, x_1, w_2\}\cup\{v_i:i\in \{3, \hdots, r-1\}\}$ where $v_i\in V_i$ for each $i\in \{3, \hdots, r-1\}$. Let $B_i = A_i\setminus \{v_i\}$ for each $i\in \{3, \hdots, r-1\}$ and let $B_1 = A_1$. Let $t' = \min_i|B_i|$. Then we may partition 
    $$G[B_1\cup\bigcup_{i=3}^{r-1}B_i] = S_1'\cup \hdots\cup S_{t'}'$$ 
    into disjoint copies of $K_{r-2}$, where, again, for $i\geq 1$, $S_i'$ consists of exactly one vertex from each part $B_1, B_3, \hdots, B_{r-1}$. Note that we have $t' = \min_i|B_i|\geq \frac{n}{r-1}-2\ell-1$. By construction, for each $i\in [t']$, we have $G_i' \coloneqq G[S_i'\cup \{w_2, x_2\}]\cong K_{r}$. Further, since our orientation $\vF_{2, r}$ is anti-directed, we know $\vK_r^2$ has both an incoming edge and an outgoing edge incident to its distinguished center vertex $v$, so no matter how $w_2x_2$ is directed, there is at least one orientation $\vG_i'$ of $G_i'$ and a directed graph isomorphism $\sigma: \vG_i'\to \vK_r^2$ as follows: we have $\sigma(w_2) = v$, the distinguished center vertex of $\vK_r^2$, and  $\vG_i'[\{w_2, x_2\}] = \vG[\{w_2, x_2\}]$.
    
    Again, there are $\binom{r}{2}$ edges in $G_i'$ and so there are $2^{\binom{r}{2}}$ ways that each $G_i'$ could be oriented in $\vG$. Thus, regardless of how the edge $w_2x_2$ is oriented, for a given $i\in [t']$ the probability that $\vG[G_i']= \vG_i'$ is at least 
    $$\frac{1}{2^{\binom{r}{2}-1}}.$$
    Therefore, the probability that there is no $i\in [t']$ such that $\vG[G_i'] = \vG_i'$ as above is at most
    $$\left(1- \frac{1}{2^{\binom{r}{2}-1}}\right)^{\frac{n}{r-1}-4\ell-1}.$$
    
    Finally putting this all together, we may summarize the preceding  as follows:
    \begin{enumerate}[label=(\alph*)]
        \item Given a quadruple of vertices $w_i, x_i, w_j, x_j$ with $1\leq i, j \leq r-1$ such that $w_i,x_i\in V_i$, $w_j, x_j\in V_j$ and $G[\{w_i, x_i, w_j, x_j\}]\cong K_4$, with probability $3/8$, $\vG[\{w_i, x_i, w_j, x_j\}]$ does not contain a triangle $\vT$ that includes the edges $(u,w_k)$ and $(x_k, u)$ where $k\in \{i, j\}$ and $u\in \{w_{q}, x_{q}\}$ with $q\in \{i, j\}\setminus\{k\}$.
        \item Assuming that $\vG$ does contain such a copy of $\vT$ as described in (a), the probability that $\vT$ cannot be extended to a copy of $\vK_r^1$ with the distinguished center vertex of $\vK_r^1$ at $u$ is at most
        $$\left(1-\frac{1}{2^{\binom{r}{2}-3}}\right)^{\frac{n}{r-1}-4\ell}.$$
        \item Assuming that $\vG$ does contain such a copy of $\vK_r^1$ as described in (b), the probability that we cannot find a copy of $\vK_r^2$, also centered at $u$ which is otherwise disjoint from $\vK_r^1$ is at most
        $$\left(1- \frac{1}{2^{\binom{r}{2}-1}}\right)^{\frac{n}{r-1}-4\ell-1}.$$
    \end{enumerate}
    
    Now, if there is no copy of $\vF_{2,r}$ centered in $\{w_i,x_i, w_j, x_j\}$, then in the list above, either: 
    \begin{itemize}
        \item (a) holds or 
        \item (a) does not hold but (b) does or
        \item neither (a) nor (b) hold but (c) does.
    \end{itemize}
    
    Therefore, given $i\neq j\in [r-1]$ such that $G[\{w_i,x_i, w_j, x_j\}]\cong K_4$, the probability that there exists no copy of $\vF_{2, r}$ centered in $\{w_i, x_i, w_j, x_j\}$ is at most
    $$\frac{3}{8} +\left(1-\frac{1}{2^{\binom{r}{2}-3}}\right)^{\frac{n}{r-1}-4\ell} +\left(1- \frac{1}{2^{\binom{r}{2}-1}}\right)^{\frac{n}{r-1}-4\ell-1} = \frac{3}{8} +\epsilon(n)< \frac{1}{2}$$
    for large enough $n$.

    Finally, there are at least $\binom{\ell}{2}-\ell+2$ quadruples of vertices $\{w_i,x_i, w_j,x_j\}$ that form a $K_4$.  Since any two copies of $K_4$ formed in this way share at most one edge, the probabilities that each of them contain some triangle $\vT$ contained in $\vK_1^r$ are independent. 
    
    Further, the probability that no copy of $\vT$ can be extended to a copy of $\vF_{2,r}$ is certainly smaller than the probability that a given copy of $\vT$ cannot be extended to $\vF_{2, r}$ which is at most:
    $$\ep \defeq \left(1-\frac{1}{2^{\binom{r}{2}-3}}\right)^{\frac{n}{r-1}-4\ell} +\left(1- \frac{1}{2^{\binom{r}{2}-1}}\right)^{\frac{n}{r-1}-4\ell-1}.$$
    Therefore, the probability that an arbitrary orientation $\vG$ of $G$ does not contain a copy of $\vF_{2, r}$ is at most
    $$\left(\frac{3}{8}\right)^{\binom{\ell}{2}-\ell+2} +\epsilon(n) < \frac{1}{2^{s-1}}$$
    for large enough $n$.
\end{proof}

\begin{proof}[Proof of Theorem \ref{k-fans_thm_2}]
Let $G$ be a graph with $n$ vertices for sufficiently large $n$ such that $D(n, \vF_{3,3}) = D(G, \vF_{3,3})$ and let $V(G) = V_1\cup V_2$ be an optimal partition of $G$. Suppose for contradiction that $e(G)>\ex(n, \vF_{3,3})$. 

Recall that the extremal graph for $F_{3,3}$ consists of a balanced complete bipartite graph with two disjoint copies of $K_{3}$ placed in a single part. By Corollary \ref{kfans_cor} we know that 
$$2^{\floor{\frac{n^2}{4}}+6}\leq D(n, F_{3,3}) \leq 2^{\floor{\frac{n^2}{4}}+12}.$$
Moreover recall that $\cM(F_{3,3}) = \cM'(\vF_{3,3})$ consists of a matching of size $3$ and a star $S_4$. Thus by Lemma \ref{parts_lem}, for $i\in [2]$, $G[V_i]$ has maximum degree $2$ and does not contain a $3$-matching, and thus $|E(V_1)|, |E(V_2)|\leq h(3) = 6$. Since we have assumed $G$ is not the extremal graph for $F_{3,3}$, we may assume that both $V_1$ and $V_2$ contain edges. Further, since $|E(G)|> \ex(n, \vF_{3,3})$, there are at most $\ell \defeq |E(V_1)| + |E(V_2)|-7  \leq 5$ edges missing from between $V_1$ and $ V_2$.

For each $i\in [2]$, let $X_i = \{v\in V_i: d_{V_i}(v)>0\}$. Since $G[X_i]$ has maximum degree 2, any connected components in $G[X_i]$ are paths or cycles. Further, since $G[X_i]$ does not have a 3-matching, this set of possible connected components is reduced to $P_2, P_3, P_4, P_5, C_3, C_4, C_5$. Also,  $G[X_i]$ can have at most two connected components. Finally, if $G[X_i]$ contains any of $P_4, P_5, C_4, C_5$ then $G[X_i]$ can have only one connected component.

Therefore, the possible $X_i$'s are:

\begin{multicols}{3}
\begin{itemize}
    \item $a_1 =P_2$
    \item $a_2 = P_3$
    \item $a_3 = P_4$
    \item $a_4 = P_5$
    \item $a_5 = C_3$
    \item $a_6 = C_4$
    \item $a_7 = C_5$
    \item $a_8 = P_2\sqcup P_3$
    \item $a_9 = P_2\sqcup C_3$
    \item $a_{10} = P_3\sqcup C_3$
    \item $a_{11} = P_2\sqcup P_2$
    \item $a_{12} = P_3\sqcup P_3$
    \item $a_{13} = C_3\sqcup C_3$
\end{itemize}
\end{multicols}

Now, using the enumeration above, each possible pair $\{G[X_1], G[X_2]\}$ can be written as a pair $\{a_i,a_j\}$. We state two claims.
\begin{claim}
\label{3_verts_claim}
Suppose there are three vertices $q_1, q_2, q_3\in G[X_i]$ each with degree 2, and at least one nonempty component $C$ in $V_{3-i}$ such that there is an edge in $E(G)$ between each vertex of $C$ and $q_1, q_2$ and $q_3$. Then $G$ does not maximize $D(G, \vF_{3,3})$.
\end{claim}
\begin{claim}
\label{paths-claim}
Suppose $G$ has at most $\floor{\frac{n^2}{4}} + 8$ edges. Suppose there are three edge-disjoint $3$-paths 
$$P_3= \{x_1, y_1, z_1\}, P'_{3}= \{x'_1, y'_1, z'_1\} \subseteq G[V_1]\text{ and } P''_{3}= \{x_2, y_2, z_2\}\subseteq G[V_2]$$ such that every possible edge edge between $\{y_1, y_1'\}$ and $P_{3}''$, as well as every possible edge between $\{y_2\}$ and $P_{3}\cup P_{3}'$ is present. Then $D(G, \vF_{3,3})< D(n,\vF_{3,3})$.
\end{claim}

%%%%%%%%%%%%%%%%%%%%%%%
\begin{table}[!b]
    \centering
    \begin{tabular}{|c||c|c|c|c|c|c|c|c|c|c|c|c|c|}
\hline
  {} & $a_1$ & $a_2$ & $a_3$ & $a_4$ & $a_5$ & $a_6$ & $a_7$ & $a_8$ & $a_9$ & $a_{10}$ & $a_{11}$ & $a_{12}$ & $a_{13}$ \\
  \hline
  $a_{1}$&  &  &   &   &   &  &   &  &  &   &  &   &    \\
  \hline
  $a_{2}$& &  &   &   &   &  &   &  &  &   &  &   &    \\
  \hline
  $a_{3}$& &  &   &   &   &  &   &  &  &   &  &   &    \\
  \hline
  $a_{4}$&  &  & 1 & 1,2 &   &  &  &   &  &   &  &   &    \\
  \hline
  $a_{5}$& &  &   &  1 &   &  &   &  &  &   &  &   &    \\
  \hline
  $a_{6}$&  &  &  1 &  1 & 1  & 1 &   &  &  &   &  &   &    \\
  \hline
  $a_{7}$& & 1 & 1  &  1 &  1 & 1 & 1,2  &  &  &   &  &   &    \\
  \hline
  $a_{8}$& &  &   & 1  &   & 1 & 1  &  &  &   &  &   &    \\
  \hline
  $a_{9}$& &  & 1  & 1  & 1  & 1 & 1  &1  & 1 &   &  &   &    \\
  \hline
  $a_{10}$& & 1 & 1  & 1,2  & 1  & 1 & 1  & 1 & 1 & 1  &  &   &    \\
  \hline
  $a_{11}$& &  &   &   &   &  & 1  &  &  & 1  &  &   &    \\
  \hline
  $a_{12}$& &  &  2 & 1  & 1  & 1 & 1  & 2 & 1 & 1  &  &  2 &    \\
  \hline
  $a_{13}$&1 & 1 & 1  & 1  &  1 & 1 & 1  & 1 & 1 & 1  & 1 & 1  &  1  \\
  \hline
\end{tabular}
    \caption{Possible pairs $\{G[X_1], G[X_2]\}$ and the claim(s) which we must apply to show $D(G, \vF_{3,3})\neq D(n, \vF_{3,3})$}
    \label{fig:enter-label}
\end{table}

We claim that for any pair $\{a_i, a_j\}$ with $i, j\in [13]$ we may apply at least one of Claim \ref{3_verts_claim} or Claim \ref{paths-claim} to show that $G$ does not maximize $D(G, \vF_{3,3})$. For each pair $\{a_i,a_j\}$ with $i\leq j$, the table below gives the number of the claim we apply to show the complete bipartite graph with $a_i$ placed in one part and $a_j$ placed in the other does not maximize $D(G, \vF_{3,3})$. For positions $a_ia_j$ in Table \ref{fig:enter-label} which contain a 1, even if $G[X_1\cup X_2]$ is missing the maximum number, $\ell$, of edges, one may check that the assumptions of Claim \ref{3_verts_claim} are satisfied. For positions $a_ia_j$ in Table \ref{fig:enter-label} containing a 1 and a 2, the conditions of Claim \ref{3_verts_claim} are satisfied if there are at most $\ell-2$ edges missing from between $X_1$ and $X_2$, and the conditions of Claim \ref{paths-claim} are satisfied if there are $\ell$ or $\ell-1$, edges missing from between $X_1$ and $X_2$. If a position $a_ia_j$ in Table \ref{fig:enter-label} only contains a $2$, the conditions of Claim \ref{paths-claim} are satisfied regardless of how many edges are missing from $G[X_1\cup X_2]$. If a cell $a_ia_j$ with $i\leq j$ is blank, the graph with the pair $\{a_i, a_j\}$ necessarily does not have $\floor{n^2/4} +7$ edges.

Thus, it only remains to prove Claims \ref{3_verts_claim} and \ref{paths-claim}, which we postpone to the appendix.

\end{proof}

\section{Other orientations of $F_{2,3}$ and examples}
We consider the orientations of $F_{2,3}$ which are not covered by Theorem \ref{k-fans_thm_1}, i.e., are not anti-directed. For ease of notation, we denote by $B$ the graph $F_{2,3}$ and call it a bowtie. Recall by Theorem \ref{k-fans_thm_1} that if $\vB$ is an orientation of $B$ such that both triangles have exactly one edge entering the center vertex and exactly one edge exiting the center vertex, then for large enough $n$, 
$$D(n, \vB) = 2^{\floor{\frac{n^2}{4}}+1}.$$
Therefore, we consider other orientations $\vB$ of $B$.

\begin{proof}[Proof of Proposition \ref{not_semi_anti_directed_prop}]
    Let $G$ be a graph defined as follows: There is one vertex $v$ is adjacent to every other vertex in $V(G)$ and $G\setminus\{v\}$ is a balanced complete bipartite graph with parts $V_1$ and $V_2$. We break into cases depending on whether $\vB$ falls into case (a) or (b).
    \begin{enumerate}[label=(\alph*)]
        \item Suppose $\vB$ is the orientation such that all edges incident to the center vertex are oriented towards the center vertex. Suppose there are two vertices in $x_1, y_1\in V_1$ and two vertices $x_2,y_2\in V_2$ such that for each $i\in [2]$ $(x_i, v),(y_i, v)\in E(\vG)$.  Then regardless of how the edges between $V_1$ and $V_2$ are oriented, these five vertices form a copy of $\vB$ with triangles $x_1x_2v$ and $y_1y_2v$. Further note that since every copy of $B$ in $G$ is centered at $v$, if there are not two vertices $x_i,y_i\in V_i$ as above, then $\vG$ does not contain $\vB$.  Thus the number of $\vB$-free orientations of $G$ is 
        $$2^{|V_1||V_2|}(|V_1|2^{|V_2|}+|V_2|2^{|V_1|}).$$
        So we have 
        $$\begin{aligned}
        D(G, \vB) &\geq 2^{\floor{\frac{(n-1)^2}{4}}}\left(\floor{\frac{n-1}{2}}2^{\ceiling{\frac{n-1}{2}}}+\ceiling{\frac{n-1}{2}}2^{\floor{\frac{n-1}{2}}}-\floor{\frac{n^2}{4}}\right)\\ 
        &= n(1-o(1))2^{\frac{n^2}{4}}.
        \end{aligned}$$
The case where $\vB$ has all four edges oriented away from the center is similar.
\item Suppose $\vB$ is the orientation of $B$ such that there are vertices $x_1, y_1\in V_1$ and $x_2, y_2\in V_2$ such that for each $i\in [2]$, $(x_i,v)\in E(G)$ and $(v, y_i)\in G$. Again, regardless of how the edges between $V_1$ and $V_2$ are oriented, these five vertices form a copy of $\vB$ with triangles $x_1x_2v$ and $y_1y_2v$. Further, if there is $i\in [2]$ such that the pair $(v, V_i)$ is pure, then there is no $\vB$ in $\vG$. Thus, the number of $\vB$-free orientations of $G$ is 
        $$2^{|V_1||V_2|}(2\cdot 2^{|{V_1|}} +2\cdot 2^{|V_2|}).$$
        So we have 
        $$D(G, \vB) \geq 2^{\floor{\frac{(n-1)^2}{4}}+1}\left(2^{\ceiling{\frac{n-1}{2}}}+2^{\floor{\frac{n-1}{2}}}-4\right) =  2^{\floor{\frac{n^2}{4}}+\frac{7}{4} - o(1)}.$$
    \end{enumerate}
\end{proof}

The only remaining orientations $\vB$ of $B$ not covered by Theorem \ref{k-fans_thm_1} or \ref{not_semi_anti_directed_prop} are those where $\vB$ has exactly three edges entering the center vertex or exactly three edges exiting the center vertex. When this is the case, one can easily check that the graph $G$ from the proof of Proposition \ref{not_semi_anti_directed_prop} has fewer than $2^{\floor{n^2/4}+1}$ many $\vB$-free orientations. Indeed, we have been unable to find any graph $G$ such that $D(G, \vB)>2^{\floor{n^2/4}+1}$.  

On the other hand, one can see that for any orientation $\vB$ of $B$ which is not anti-directed, $\cM(\vB)$ only contains at most one of the pure orientations of $S_3$. Therefore, we may not apply Theorem \ref{general_graph_theorem} or \ref{general_gph_finite_thm}. Thus, for any orientation $\vB$ which is not anti-directed we have no good upper bound on $D(n, \vB)$.

Next, we have an example of a graph that $H$ such that no orientation $\vH$ of $H$ has the property that $D(n, \vH) = 2^{\ex(n, H)}$. Let $W_{2k+1}$ represent the wheel graph with $2k+1$ vertices. Suppose $k$ is even and sufficiently large.  We claim that for any orientation $\vW_{2k+1}$ of $W_{2k+1}$, and $n_0$, there is some $n>n_0$ such that
$$D(n, \vW_{2k+1})>2^{\ex(n, W_{2k+1})}.$$

\begin{proof}[Proof of Proposition \ref{wheel_prop}]
    In \cite{YUAN2021} Yuan showed that for $k\geq 3$, 
    $$\ex(n, W_{2k+1}) = \max\{n_0n_1 + \floor{\frac{(k-1)n_0}{2}}:n_0 + n_1 = n\}$$
    This graph was obtained by letting $G$ be a bipartite graph with parts $|V_0| = n_0$, $|V_1| = n_1$ and placing a $(k-1)$-regular or almost $(k-1)$-regular graph in one part (depending on the parity of $k, n_0$) and a single edge in the other part.  Let $n$ be such that an extremal graph for $W_{2k+1}$ has an almost $(k-1)$-regular graph in $V_0$.  Then there is a component of $G[V_0]$ with fewer than $2k$ vertices such that all but one vertex have degree $k-1$ and the final vertex, call it $w$, has degree $k-2$.  Let $v$ be another vertex in that component, which is not a neighbor of $w$. Let $G' = G\cup\{vw\}$.  We claim that for any orientation $\vW_{2k+1}$, 
    $$D(G', \vW_{2k+1})>D(G, \vW_{2k+1}) = 2^{\ex(n, W_{2k+1})}.$$ 
    Indeed, we only have to show that if the edges of $G'$ are directed uniformly at random, the probability of there being a copy of $\vW_{2k+1}$ is less than $\frac{1}{2}$.  
    
    Notice that $\cM(W_{2k+1}) = \{S_{k+1}, C_{2k}\}$. Further, $\cM(\vW_{2k+1})$ only contains two orientations of $S_{k+1}$, (given by taking the orientations of every other spoke of $\vW_{2k+1}$ in the two possible ways), and one orientation of $C_{2k}$.  Now, $G'[V_0]$ still contains no $2k$-cycles and there is exactly one vertex in $G'[V_0]$ of degree at least $k$, namely $v$. Thus, every copy of $W_{2k+1}$ in $G'$ is centered at $v$.  There are $2^k$ ways to orient the $S_{k+1}$ centered at $v$ contained in $V_0$, and for any given orientation of $\vS_{k+1}$, at most $\binom{k}{k/2}$ of these orientations give rise to an orientation isomorphic to $\vS_{k+1}$. Thus the probability of there being a copy of $\vW_{2k+1}$ contained in an orientation $\vG'$ is at most $\frac{2\binom{k}{k/2}}{2^k}<\frac{1}{2}$ since $k$ is sufficiently large. So in particular,
    $$D(G', \vW_{2k+1})>\frac{1}{2}2^{|E(G')|}= \frac{1}{2}2^{\ex(n, W_{2k+1} + 1)}.$$
\end{proof}

\section{Acknowledgments}
We would like to thank Caroline Terry for her help and advice throughout the entirety of this project.

\appendix

\section{Proof of Lemma \ref{main_lemma}}\label{Proof_of_Lemma}
We begin by setting some notation. Let $\vH$ be a directed graph. 
 Given two subsets $A, B\subseteq V(\vH)$ we define the density of edges to be
$$d(A, B) = \frac{|E(A, B)|}{|A||B|}.$$  
Note that since $\vH$ is a directed graph, $E(A, B)\neq E(B, A)$ so in general $d(A, B)\neq d(B, A)$.  Let $\epsilon>0$. Given $A, B\subseteq V(G)$, the pair $(A, B)$ is called $\epsilon$-regular if for every $X\subseteq A$, $Y\subseteq B$ with $|X|\geq \epsilon|A|$ and $|Y|\geq \epsilon |B|$, we have $|d(X, Y)-d(A, B)|\leq \epsilon$ and $|d(Y, X)-d(B, A)|\leq \epsilon$. An equitable partition is a partition of the vertices of $G$ into subsets of as equal size as possible.  An equitable partition $V(G) = V_1\cup \hdots\cup V_m$ is called $\epsilon$-regular if $|V_i|\leq \epsilon |V(G)|$ for each $i$, and all but at most $\epsilon \binom{m}{2}$ of the pairs $(V_i, V_j)$ are $\epsilon$-regular.
The following lemma is the directed regularity lemma, a varient of Szemeredi's regularity lemma.

\begin{lemma}[Alon, Yuster \cite{Alon2006}]
For every $\epsilon>0$ there exists $M$ such that any graph on $n>M$ vertices has an $\epsilon$-regular partition of the vertex set $V = V_1\cup\hdots\cup V_m$ for some $1/\epsilon\leq m\leq M$.
\end{lemma}

\begin{proof}[Proof of Lemma \ref{main_lemma}]
Let $G$ be a graph with at least $2^{t(n,p)}$ $\vH$-free orientations. We begin by fixing $0<\epsilon\leq \eta\leq\beta\leq\alpha\leq\delta$. Let $\vG$ be an orientation of $G$.  We apply the directed regularity lemma to $\vG$ to obtain an $\epsilon$-regular partition $V(G) = V_1\cup\hdots\cup V_m$. We define a directed graph $C$, called the cluster graph, as follows: the vertices of $C$ are $V_1, \hdots , V_m$, and the there is an edge from $V_i$ to $V_j$ if $(V_i, V_j)$ is an $\epsilon$-regular pair, and the density $d(V_i, V_j)>\eta$.  Note this graph may have bidirected edges. 

We will show that for some $\vH$-free orientation $\vG$ of $G$, the graph $C$ has at least $(1-\frac{1}{p})\frac{m^2}{2}-\beta m^2$ many bidirected edges.  Suppose not. We will draw the contradiction that $G$ does not have $2^{t(n,p)}$ orientations (and so $G$ certainly does not have $2^{t(n,p)}$ $\vH$-free orientations).  Let $M = M(\epsilon)$ be as in the directed regularity lemma, and notice that $m\leq M$. Then there are fewer than $M^n$ choices for the partition $V_1,\hdots, V_m$, at most $2^{\binom{M}{2}}$ choices for the set of $\epsilon$-regular pairs in $C$ and $4^{\binom{M}{2}}$ choices for the edges in $C$.  All together, we can choose $V_1, \hdots, V_m$, our $\epsilon$-regular pairs and the edges of $C$ in at most $M^n2^{\binom{M}{2}}4^{\binom{M}{2}}\leq M^n2^{3M^2/2}$ ways. Next, given a choice of $V_1, \hdots, V_m$, the $\epsilon$-regular pairs and the edges of $C$, we bound the number of $\vH$-free orientations of $G$ which give rise to that fixed choice.  

    %1
    There are at most $\epsilon n^2$ edges inside the $V_i$'s and there are at most $\epsilon n^2$ edges between non $\epsilon$-regular pairs.  Each of these edges can be oriented in two ways so there are a total of at most $2^{2\epsilon n^2}$ ways to orient the edges that are either within one of the $V_i$'s or between a non $\epsilon$-regular pair.

    %2
    Let $(V_i, V_j)$ be an $\epsilon$-regular pair which is not an edge of $C$.  Then there are at most about $\frac{\eta n^2}{m^2}$ directed edges of $G$ from $V_i$ to $V_j$.  Therefore, there exists a constant $c_\eta$, with $c_\eta\to 0$ as $\eta\to 0$, such that there are at most $2^{c_\eta n^2/m^2}$ ways to orient these edges. Now, there are at most $m^2$ of these pairs $(V_i, V_j)$, so there are at most $2^{c_\eta n^2}$ ways to orient all the edges between such pairs.
    
    %3
    Finally, let $(V_i, V_j)$ be a pair of vertices in $C$ with a bidirected edge between them.  Then there are at most about $2^{n^2/m^2}$ orientations of the edges between these parts.  However we are assuming any $C$ has at most $(1-\frac{1}{p})\frac{m^2}{2}-\beta m^2$ bidirected edges, so the edges between all $V_i$ and $V_j$ such that $(V_i,V_ej)$ a bidirected edge in $C$ can be oriented in at most 
    $$2^{\frac{n^2}{m^2}( (1-\frac{1}{p})\frac{m^2}{2}-\beta m^2)} = 2^{(1-\frac{1}{p})\frac{n^2}{2}-\beta n^2}$$ 
    ways.
    
Putting this all together, the number of orientations of the edges of $G$ is at most 
$$M^n2^{3M^2/2}\cdot2^{2\epsilon n^2}\cdot2^{c_\eta n^2}\cdot 2^{(1-\frac{1}{p})\frac{n^2}{2}-\beta n^2} = M^n2^{3M^2/2}\cdot 2^{n^2(2\epsilon +c_\eta+\frac{1}{2}(1-\frac{1}{p})-\beta)}.$$
Choosing $\eta$ and $\epsilon$ small enough with respect to $\beta$ gives that the total number of orientations is less than $2^{t(n,p)}$, a contradiction.  Thus, there is an orientation of the edges of $G$ such that the resulting cluster graph has at least $\frac{m^2}{4}-\beta m^2$ bidirected edges.

Now fix an $\vH$-free orientation $\vG$ of $G$ such that $C$ has at least $(1-\frac{1}{p})\frac{m^2}{2}-\beta m^2$ bidirected edges.  Notice, the graph formed by only the bidirected edges of $C$ has no copy of $K_{\chi(H)}$.  Otherwise, by the embedding lemma, $\vG$ would have a copy of $\vH$.  So, the graph consisting of only bidirected edges of $C$ is $K_{\chi(H)}$-free and has at least $(1-\frac{1}{p})\frac{m^2}{2}-\beta m^2$ edges.  By the stability theorem of Simonovits there is a partition $V(C) = W_1\cup\hdots\cup W_p$ of the vertices of $C$ such that there are at most $\alpha m^2$ bidirected edges within a part.  This also implies that there are at least $(1-\frac{1}{p})\frac{m^2}{2}-(\beta + \alpha)m^2$ bidirected edges between parts $W_1, \hdots, W_p$.  Next, suppose for contradiction $C$ had more than $4p(\alpha + \beta)m^2$ (not necessarily bidirected) edges inside parts.  Then at least $4(\alpha +\beta)m^2$ of these edges are in a single part, say $W_1$.  So there is a bipartite graph with at least $2(\alpha + \beta)m^2$ edges inside $W_1$.  Noting that these edges may be bidirected, there are at least $(\alpha + \beta)m^2$  distinct pairs of vertices in $W_1$ that form a directed edge in $C$.  Now, taking these $(\alpha + \beta)m^2$ edges along with all the bidirected edges between $W_i$ and $W_j$, with $i\neq j \in [p]$, we have a graph with more than $(1-\frac{1}{p})\frac{m^2}{2}$ edges, and thus by Tur\`{a}n's theorem, contains a copy of $K_p$.  Therefore we have $|E(W_1)| +\hdots + |E(W_p)|\leq \alpha m^2 + 4p(\alpha + \beta)m^2$. 

Suppose we remove all edges of $G$ corresponding to edges within $W_1$ or $W_2$. This at most $(\alpha m^2 + 4p(\alpha + \beta)m^2)\frac{n^2}{m^2}$ edges.  In addition, remove all edges of $G$ which fall within a $V_i$, or between $V_i, V_j$ where $(V_i, V_j)$ is not an edge of $C$.  Then the resulting subgraph of $G$ is $p$-partite. There are at most $\epsilon n^2$ edges that are entirely within a $V_i$ and at most $\epsilon n^2$ edges that are between non $\epsilon$-regular pairs of our partition of $G$.  Furthermore, there are at most $2\eta n^2$ edges between $\epsilon$-regular pairs which are not edges of $C$.  Therefore, we have removed at most 
$$\alpha n^2 + 4p(\alpha + \beta)n^2 + \epsilon n^2 + \epsilon n^2 + 2\eta n^2<\delta n^2$$
edges. Thus the vertices of $G$ have a $p$-partition with at most $\delta n^2$ edges within parts.
\end{proof}

\section{Proof of Claims \ref{3_verts_claim} and \ref{paths-claim} from Theorem \ref{k-fans_thm_2}}
\begin{proof}[Proof of Claim \ref{3_verts_claim}]
    Suppose without loss of generality that $q\in V_1$ has degree 2 in $G[X_1]$, and $Y\subseteq X_2$ is a component of $G[X_2]$. Let $c = |E(Y)|$. 

    Let $\vT_1, \vT_2, \vT_3$ be the three triangles of $\vF_{3,3}$. Let $r, s$ be neighbors of $q$ in $X_1$.  Then the size of the common neighborhood of $q, r, s$ is at least $\frac{n}{2}- 5$. For any two vertices $x, y\in N(\{q, r, s\})$, $qxr$ and $qys$ are both triangles. For each of these triangles, since $\vF_{3,3}$ is anti-directed, regardless of how the edges $qr$ and $qs$ are oriented, there is at least one way to orient the edges $xs, xq, yq, yr$ such that these two triangles are isomorphic to $\vT_1$ and $\vT_2$ respectively, both centered at $q$. The probability that there is no such pair of triangles $qxr, qys$ isomorphic to $\vT_1, \vT_2$ is at most 
    $$\epsilon\defeq \left(\frac{15}{16}\right)^{n/5}$$
    which tends to $0$ as $n\to \infty$. Let $\vT_3 = tuv$ be the orientation of the remaining triangle in $\vF_{3,3}$ centered at $u$. Suppose there is a copy of $\vT_3$ centered at $q \in \{q_1, q_2, q_3\}$ with its other two vertices in $Y$. Then the probability that there is a copy of $\vF_{3,3}$ centered at $q$ is $1-\epsilon(n)$. 

    Let $p$ be the probability that an orientation of $G[\{q_1,q_2,q_3\}\cup Y]$ chosen uniformly at random contains a copy of $\vT_3$ centered at $q_i$ for some $i\in [3]$, and such that the other two vertices of this $\vT_3$ lie in $Y$. Then, given an orientation $\vG$ of $G$ chosen uniformly at random, the probability that there is a copy of $\vF_{3,3}$ centered at $q_i$ for some $i\in [3]$ and containing exactly one edge in $Y$ is at least $p(1-\epsilon) \geq p-\epsilon$. Therefore, if $1-p<1/2^c$, we may delete the edges of $Y$ from $G$ to obtain a graph with more $\vF_{3,3}$-free orientations than $\vG$. Thus, it only remains to show $1-p<\frac{1}{2^c}$ for all possible components $Y$. Note that for each triangle $quv$ with $q = q_i$ for some $i\in[3]$, and $u, v\in V(Y)$, given an orientation of the edge $uv$, there is exactly one way to orient the edges $qu, qv$ such that $G[\{q, u, v\}]\cong \vT_3$.  
    \begin{enumerate}[label=(\alph*)]
        \item Suppose $Y$ is a single edge $uv$. Then regardless of how the edge $uv$ is oriented there is exactly one way to orient the edges $qu, qv$ to form $\vT_3$. Thus, the probability that the triangle $quv$ is $\vT_3$ is $\frac{1}{4}$. The probability that $\vG[q_i,u,v] \not\cong \vT_3$ for any $i\in[3]$ is
        $$\left(\frac{3}{4}\right)^3 = \frac{27}{64} <\frac{1}{2}.$$
        \item Suppose $Y$ is $P_3 = uvw$. Label the triangles as follows: $A = quv$, and $B = qvw$. There are two ways, up to symmetry and reversing of directions, to orient the edges $uv, vw$. 
        \begin{enumerate}[label=(\roman*)]
            \item Suppose without loss of generality that 
            $$(u, v), (v,w)\in E(\vG).$$
            Suppose $qu$ is oriented so that it is impossible to have $\vG[A]\cong \vT_3$. Then the probability that $\vG[B]\cong \vT_3$ is $1/4$. If $qu$ is oriented so that we may have $\vG[A]\cong\vT_3$, then the probability that $\vG[A]\cong \vT_3$ is $1/2$. If $qv$ is oriented so that $\vG[A]\not\cong \vT_3$, then the probability that $\vG[B]\cong \vT_3$ is $1/2$. Thus the probability that either $\vG[A]$ or $\vG[B]$ is $\vT_3$ is 
            $$\frac{1}{2}\cdot\frac{1}{4} + \frac{1}{2}\cdot\frac{1}{2} + \frac{1}{2}\cdot\frac{1}{2}\cdot\frac{1}{2} = \frac{1}{2}.$$
            \item Suppose without loss of generality that 
            $$(u,v), (w,v)\in E(\vG).$$
            Suppose $qu$ is oriented so that it is impossible to have $\vG[A]\cong \vT_3$. Then the probability that $\vG[B]\cong \vT_3$ is $1/4$.  If $qu$ is oriented so that we may have $\vG[A]\cong\vT_3$, then the probability that $\vG[A]\cong \vT_3$ is $1/2$.  If $qv$ is oriented so that $\vG[A]\not\cong \vT_3$ then $\vG[B]\not\cong \vT_3$.
            $$\frac{1}{2}\cdot\frac{1}{4} + \frac{1}{2}\cdot\frac{1}{2} = \frac{3}{8}.$$
        \end{enumerate}
        Thus, since there are three possible vertices that could be $q$, the probability that neither $\vG[A]$ nor $\vG[B]\cong \vT_3$ for any $q = q_i$ is 
        $$\frac{1}{2}\left(\frac{1}{2}\right)^3 + \frac{1}{2}\left(\frac{5}{8}\right)^3 \approx 0.19<\frac{1}{2^2}$$
        \item Suppose $Y$ is $P_4 = uvwx$. Label the triangles as follows: $A = quv$, and $B = qvw$, $C = qwx$. Then there are three ways, up to symmetry and reversing of directions, to orient the edges of $Y$.
        \begin{enumerate}[label=(\roman*)]
            \item Suppose without loss of generality that 
            $$(u,v), (v,w), (w, x)\in E(\vG).$$
            Suppose $qu$ is oriented so that it is impossible to have $\vG[A]\cong \vT_3$. Then the probability that $\vG[B]$ or $\vG[C]\cong \vT_3$ is $1/2$ by case (b)(i). If $qu$ is oriented so that we may have $\vG[A]\cong\vT_3$, then the probability that $\vG[A]\cong \vT_3$ is $1/2$. If $qv$ is oriented so that $\vG[A]\not\cong \vT_3$, then the probability that $\vG[B]\cong \vT_3$ is $1/2$. If $\vG[B]\not\cong \vT_3$, then the probability that $\vG[C]\cong \vT_3$ is $1/2$. Thus the probability that either $\vG[A]$, $\vG[B]$ or $\vG[C]$ is $\vT_3$ is 
            $$\frac{1}{2}\cdot\frac{1}{2} + \frac{1}{2}\cdot\frac{1}{2} + \frac{1}{2}\cdot\frac{1}{2}\cdot\frac{1}{2} + \frac{1}{2}\cdot\frac{1}{2}\cdot\frac{1}{2}\cdot \frac{1}{2} = \frac{11}{16}.$$
            \item Suppose without loss of generality that 
            $$(v,u), (v,w), (w, x)\in E(\vG)$$
            Suppose $qu$ is oriented so that it is impossible to have $\vG[A]\cong \vT_3$. Then the probability that $\vG[B]$ or $\vG[C]\cong \vT_3$ is $1/2$.  If $qu$ is oriented so that we may have $\vG[A]\cong\vT_3$, then the probability that $\vG[A]\cong \vT$ is $1/2$. If $qv$ is oriented so that $\vG[A]\not\cong \vT_3$ then $\vG[B]\not\cong \vT_3$, and the probability that $\vG[C]\cong \vT_3$ is $1/4$. Thus the probability that either $\vG[A]$, $\vG[B]$ or $\vG[C]$ is $\vT_3$ is
            $$\frac{1}{2}\cdot\frac{1}{2} + \frac{1}{2}\cdot\frac{1}{2} + \frac{1}{2}\cdot\frac{1}{2}\cdot\frac{1}{4} = \frac{9}{16}.$$
            \item Suppose without loss of generality that 
            $$(u,v), (w,v), (w, x)\in E(\vG)$$
             Suppose $qu$ is oriented so that it is impossible to have $\vG[A]\cong \vT_3$. Then the probability that $\vG[B] $ or $\vG[C]\cong \vT_3$ is $3/8$.  If $qu$ is oriented so that we may have $\vG[A]\cong \vT_3$, then the probability that $\vG[A]\cong \vT_3$ is $1/2$. If $qv$ is oriented so that $\vG[A]\not\cong \vT_3$ then $\vG[B]\not\cong \vT_3$, so the probability that $\vG[C]\cong \vT_3$ is $1/4$. Thus the probability that either $\vG[A]$, $\vG[B]$ or $\vG[C]$ is $\vT_3$ for any $q = q_i$ is
            $$\frac{1}{2}\cdot\frac{1}{2} + \frac{1}{2}\cdot\frac{1}{2} + \frac{1}{2}\cdot\frac{1}{2}\cdot\frac{1}{4} = \frac{1}{2}.$$
        \end{enumerate}
        There are two orientations of $Y$ that fall into case (i), four that fall into case (ii), and two that fall into case (iii). Thus, since there are three possible vertices that could be $q$, the probability that neither $\vG[A], \vG[B]$ nor $\vG[C]\cong \vT_3$ is 
        $$\frac{1}{4}\left(\frac{5}{16}\right)^3 + \frac{1}{2}\left(\frac{7}{16}\right)^3 + \frac{1}{4}\left(\frac{1}{2}\right)^3 \approx 0.081<\frac{1}{2^3}.$$
        \item Suppose $Y$ is $P_5 = uvwxy$. Label the triangles as follows: $A = quv$, and $B = qvw$, $C = qwx$, $D = qxy$. Then there are six ways, up to symmetry and reversing of directions, to orient the edges of $Y$.
        \begin{enumerate}[label=(\roman*)]
            \item Suppose without loss of generality that 
            $$(u,v), (v,w), (w, x), (x,y)\in E(\vG).$$
            Using an argument similar to that in case (c)(i), we have that the probability that $\vG[A], \vG[B], \vG[C],$ or $\vG[D]\cong \vT_3$ is 
            $$\frac{1}{2}\cdot \frac{11}{16} + \left(\frac{1}{2}\right)^2 + \left(\frac{1}{2}\right)^3 + \left(\frac{1}{2}\right)^4 + \left(\frac{1}{2}\right)^5 = \frac{13}{16}.$$
            \item Suppose without loss of generality that 
            $$(u,v), (v,w), (w, x), (y,x)\in E(\vG).$$
            By an argument similar that in case (b)(ii), we have that the probability that $\vG[A], \vG[B], \vG[C]$ or $\vG[D]\cong \vT_2$ is 
            $$\frac{1}{2}\cdot\frac{11}{16} + \frac{1}{2}\cdot\frac{1}{2} + \frac{1}{2}\cdot\frac{1}{2}\cdot\frac{1}{2} = \frac{23}{32}.$$
            \item Suppose without loss of generality that 
            $$(u,v), (v,w), (x, w), (x,y)\in E(\vG).$$
            By an argument similar that in case (b)(ii), we have that the probability that $\vG[A], \vG[B], \vG[C]$ or $\vG[D]\cong \vT_3$ is 
            $$\frac{1}{2}\cdot\frac{9}{16} + \frac{1}{2}\cdot\frac{1}{2} + \frac{1}{2}\cdot\frac{1}{2}\cdot\frac{1}{2} = \frac{21}{32}.$$
            \item Suppose without loss of generality that 
            $$(u,v), (v,w), (x, w), (x,y)\in E(\vG).$$
            By an argument similar that in case (b)(ii), we have that the probability that $\vG[A], \vG[B], \vG[C]$ or $\vG[D]\cong \vT_3$ is 
            $$\frac{1}{2}\cdot\frac{9}{16} + \frac{1}{2}\cdot\frac{1}{2} + \frac{1}{2}\cdot\frac{1}{2}\cdot\frac{1}{2} = \frac{21}{32}.$$
            \item Suppose without loss of generality that 
            $$(u,v), (v,w), (x, w), (y,x)\in E(\vG).$$
            Suppose $qu$ is oriented so that it is impossible to have $\vG[A]\cong \vT_3$. Then the by case (b)(ii) probability that $\vG[B], \vG[C]$ or $\vG[D]\cong \vT_3$ is $9/16$.  If $qu$ is oriented so that we may have $\vG[A]\cong \vT_3$, then the probability that $\vG[A]\cong \vT_3$ is $1/2$. If $qv$ is oriented so that $\vG[A]\not\cong \vT_3$ then the probability that $\vG[B]\cong \vT_3$ is $1/2$. If $qw$ is oriented so that $\vG[B]\not\cong \vT_3$, then $\vG[C]\not\cong\vT_3$ so the probability that $\vG[D]\cong \vT_3$ is $1/4$.
            $$\frac{1}{2}\cdot\frac{9}{16} + \frac{1}{2}\cdot\frac{1}{2} + \frac{1}{2}\cdot\frac{1}{2}\cdot\frac{1}{2} + \frac{1}{2}\cdot\frac{1}{2}\cdot\frac{1}{2}\cdot\frac{1}{4}= \frac{11}{16}.$$
            \item Suppose without loss of generality that 
            $$(u,v), (w,v), (w, x), (y,x)\in E(\vG).$$
            By an argument similar that in case (b)(ii), we have that the probability that $\vG[A], \vG[B], \vG[C]$ or $\vG[D]\cong \vT_3$ is 
            $$\frac{1}{2}\cdot\frac{1}{2} + \frac{1}{2}\cdot\frac{1}{2} + \frac{1}{2}\cdot\frac{1}{2}\cdot\frac{3}{8} = \frac{19}{32}.$$
        \end{enumerate}
        There are two orientations of $Y$ which fall into each of cases (i), (iv), (v), and (vi), and four orientations of $Y$ which fall into each of cases (ii) and (iii). Thus, since there are three possible vertices that could be $q$, the probability that neither $\vG[A], \vG[B], \vG[C]$ nor $\vG[D]\cong \vT_3$ for any $q= q_i$ is 
        $$\begin{aligned}\frac{1}{8}\left(\frac{3}{16}\right)^3 &+ \frac{1}{4}\left(\frac{9}{32}\right)^3 + \frac{1}{4}\left(\frac{11}{32}\right)^3 + \frac{1}{8}\left(\frac{11}{32}\right)^3 + \frac{1}{8}\left(\frac{5}{16}\right)^3 + \frac{1}{8}\left(\frac{13}{32}\right)^3\\
        &\approx 0.034 < \frac{1}{2^4}.
        \end{aligned}$$
         \item Suppose $Y$ is $C_3 = uvw$. Label the triangles as follows: $A = quv$, and $B = qvw$, $C = qwu$. Then there are two ways, up to symmetry and reversing of directions, to orient the edges of $Y$.
        \begin{enumerate}[label=(\roman*)]
            \item Suppose without loss of generality that 
            $$(u,v), (v,w), (w, u)\in E(\vG).$$
            Notice that as long as there are edges $(q, t_1), (t_2, q)\in E(\vG)$, then at least one of $\vG[A], \vG[B]$ or $\vG[C]\cong \vT_3$. Thus the probability of this is 
            $$\frac{3}{4}.$$
            \item Suppose without loss of generality that 
            $$(u,v), (w,v), (w, u)\in E(\vG).$$
            Suppose $qv$ is oriented so that it is impossible to have $\vG[A]\cong \vT_3$. Then $\vG[B]\not\cong \vT_3$ as well. The probability that $\vG[C]\cong \vT_3$ is $1/4$. Thus suppose $qv$ is oriented in such a way that we may have have $\vG[A]\cong \vT_3$ or $\vG[B]\cong \vT_3$. Then the probability that $\vG[A]\cong \vT_3$ or $\vG[B]\cong \vT_3$ is $3/4$. If $\vG[A]\not\cong\vT_3$ and $\vG[B]\not\cong \vT_3$, then $\vG[C]\not\cong \vT_3$.  Thus, the probability that at least one of $\vG[A], \vG[B],$ or $\vG[C]\cong \vT_3$ is 
            $$\frac{1}{2}\cdot \frac{1}{4} + \frac{1}{2}\cdot \frac{3}{4} = \frac{1}{2}.$$
        \end{enumerate}
        There are two orientations of $Y$ which fall into case (i) and six which fall into case (ii). Thus, since there are three possible vertices that could be $q$, the probability that neither $\vG[A], \vG[B]$ nor $\vG[C]\cong \vT_3$ for any $q=q_i$ is 
        $$\frac{1}{4}\left(\frac{1}{4}\right)^3 + \frac{3}{4}\left(\frac{1}{2}\right)^3\approx 0.098<\frac{1}{2^3}.$$
        \item Suppose $Y$ is $C_4 = uvwx$. Label the triangles as follows: $A = quv$, and $B = qvw$, $C = qwx$, $D = qxu$. Then there are four ways, up to symmetry and reversing of directions, to orient the edges of $Y$.
        \begin{enumerate}[label=(\roman*)]
            \item Suppose without loss of generality that 
            $$(u,v), (v,w), (w, x), (x,u)\in E(\vG).$$
            Notice that as long as there are edges $(q, t_1), (t_2, q)\in E(\vG)$, then at least one of $\vG[A], \vG[B]$ or $\vG[C]\cong \vT_3$. Thus the probability of this is 
            $$\frac{7}{8}.$$
            \item Suppose without loss of generality that 
            $$(u,v), (w,v), (w, x), (x,u)\in E(\vG).$$
            Suppose $qv$ is oriented so that it is impossible to have $\vG[A]\cong \vT_3$. Then $\vG[B]\not\cong \vT_3$ as well. The probability that $\vG[C]$ or $\vG[D]\cong \vT_3$ is $1/2$. Thus suppose $qv$ is oriented in such a way that we may have $\vG[A]\cong \vT_3$ or $\vG[B]\cong \vT_3$. Then the probability that $\vG[A]\cong \vT_3$ or $\vG[B]\cong \vT_3$ is $3/4$. If $\vG[A]\not\cong\vT_3$ and $\vG[B]\not\cong \vT_3$, then $\vG[C]\not\cong \vT_3$. The probability that $\vG[D]\cong\vT_3$ is $1/2$.  Thus, the probability that at least one of $\vG[A], \vG[B], \vG[C]$ or $\vG[D]\cong \vT_3$ is 
            $$\frac{1}{2}\cdot \frac{1}{2} + \frac{1}{2}\cdot \frac{3}{4} + \frac{1}{2}\cdot\frac{1}{4}\cdot\frac{1}{2} = \frac{11}{16}.$$
            \item Suppose without loss of generality that 
            $$(u,v), (w,v), (x,w), (x, u)\in E(\vG).$$
            Suppose $qv$ is oriented so that it is impossible to have $\vG[A]\cong \vT_3$. Then $\vG[B]\not\cong \vT_3$ as well. The probability that $\vG[C]$ or $\vG[D]\cong \vT_3$ is $3/8$. Thus suppose $qv$ is oriented in such a way that we may have $\vG[A]\cong \vT_3$ or $\vG[B]\cong \vT_3$. Then the probability that $\vG[A]\cong \vT_3$ or $\vG[B]\cong \vT_3$ is $3/4$. If $\vG[A]\not\cong\vT_3$ and $\vG[B]\not\cong \vT_3$, the probability that $\vG[C]$ and $\vG[D]\cong\vT_3$ is $1/2$.  Thus, the probability that at least one of $\vG[A], \vG[B], \vG[C]$ or $\vG[D]\cong \vT_3$ is 
            $$\frac{1}{2}\cdot \frac{3}{8} + \frac{1}{2}\cdot \frac{3}{4} + \frac{1}{2}\cdot\frac{1}{4}\cdot\frac{1}{2} = \frac{5}{8}.$$
            \item Suppose without loss of generality that 
            $$(u,v), (w, v), (x, w), (u, x)\in E(\vG).$$
            Suppose $qv$ is oriented so that it is impossible to have $\vG[A]\cong \vT_3$. Then $\vG[B]\not\cong \vT_3$ as well. The probability that $\vG[C]$ or $\vG[D]\cong \vT_3$ is $3/8$. Thus suppose $qv$ is oriented in a way that we may have $\vG[A]\cong \vT_3$ or $\vG[B]\cong \vT_3$. Then the probability that $\vG[A]\cong \vT_3$ or $\vG[B]\cong \vT_3$ is $3/4$. If $\vG[A]\not\cong\vT_3$ and $\vG[B]\not\cong \vT_3$, then we have $\vG[C]$ and $\vG[D]\not\cong \vT_3$.  Thus, the probability that at least one of $\vG[A], \vG[B], \vG[C]$ or $\vG[D]\cong \vT_3$ is 
            $$\frac{1}{2}\cdot \frac{3}{8} + \frac{1}{2}\cdot \frac{3}{4}= \frac{9}{16}.$$
        \end{enumerate}
        There are two orientations of $Y$ that fall into each of cases (i) and (iv), four which fall into case (iii) and eight which fall into case (ii). Thus, since there are three possible vertices that could be $q$, the probability that neither $\vG[A], \vG[B], \vG[C]$ nor $\vG[D]\cong \vT_3$ for any $q = q_i$ is 
        $$\frac{1}{8}\left(\frac{1}{8}\right)^3 + \frac{1}{2}\left(\frac{5}{16}\right)^3 + \frac{1}{4}\left(\frac{3}{8}\right)^3 + \frac{1}{8}\left(\frac{7}{16}\right)^3\approx 0.039<\frac{1}{2^4}.$$
        \item Suppose $Y$ is $C_5 = uvwxy$. Label the triangles as follows: $A = quv$, $B = qvw$, $C = qwx$, $D = qxy$, and $E = qyu$. Then there are four ways, up to symmetry and reversing of directions, to orient the edges $Y$.
    \begin{enumerate}[label=(\roman*)]
            \item Suppose without loss of generality that 
            $$(u,v), (v,w), (w, x), (x,y), (y,u)\in E(\vG).$$
            Notice that as long as there are edges $(q, t_1), (t_2, q)\in E(\vG)$, then at least one of $\vG[A], \vG[B]$ or $\vG[C]\cong \vT_3$. Thus the probability of this is 
            $$\frac{15}{16}.$$
            \item Suppose without loss of generality that 
            $$(u,v), (w,v), (w, x), (x,y), (y,u)\in E(\vG).$$
            Suppose $qv$ is oriented so that it is impossible to have $\vG[A]\cong \vT_3$. Then $\vG[B]\not\cong \vT_3$ as well. The probability that $\vG[C], \vG[D]$ or $\vG[E]\cong \vT_3$ is $11/16$. Thus suppose $qv$ is oriented in a way that we may have $\vG[A]\cong \vT_3$ or $\vG[B]\cong \vT_3$. Then the probability that $\vG[A]\cong \vT_3$ or $\vG[B]\cong \vT_3$ is $3/4$. If $\vG[A]\not\cong\vT_3$ and $\vG[B]\not\cong \vT_3$, then $\vG[C]\not\cong \vT_3$. The probability that $\vG[E]\cong\vT_3$ is $1/2$. If $\vG[E]\not\cong \vT_3$, then the probability that $\vG[D]\cong\vT_3$ is $1/2$.  Thus, the probability that at least one of $\vG[A], \vG[B], \vG[C], \vG[D]$ or $\vG[E]\cong \vT_3$ is 
            $$\frac{1}{2}\cdot \frac{11}{16} + \frac{1}{2}\cdot \frac{3}{4} + \frac{1}{2}\cdot\frac{1}{4}\cdot\frac{1}{2} + \frac{1}{2}\cdot\frac{1}{4}\cdot\frac{1}{2}\cdot\frac{1}{2} = \frac{13}{16}.$$
            \item Suppose without loss of generality that 
            $$(u,v), (w,v), (x,w), (x, y), (y,u)\in E(\vG).$$
            Suppose $qv$ is oriented so that it is impossible to have $\vG[A]\cong \vT_3$. Then $\vG[B]\not\cong \vT_3$ as well. The probability that $\vG[C], \vG[D]$ or $\vG[E]\cong \vT_3$ is $9/16$. Thus suppose $qv$ is oriented in a way that we may have $\vG[A]\cong \vT_3$ or $\vG[B]\cong \vT_3$. Then the probability that $\vG[A]\cong \vT_3$ or $\vG[B]\cong \vT_3$ is $3/4$. If $\vG[A]\not\cong\vT_3$ and $\vG[B]\not\cong \vT_3$, the probability that $\vG[C]$ or $\vG[E]\cong\vT_3$ is $3/4$. If $\vG[C], \vG[E]\not\cong \vT_3$, then $\vG[D]\not\cong \vT_3$.  Thus, the probability that at least one of $\vG[A], \vG[B], \vG[C], \vG[D]$ or $\vG[E]\cong \vT_3$ is 
            $$\frac{1}{2}\cdot \frac{9}{16} + \frac{1}{2}\cdot \frac{3}{4} + \frac{1}{2}\cdot\frac{1}{4}\cdot\frac{3}{4} = \frac{3}{4}.$$
            \item Suppose without loss of generality that 
            $$(u,v), (w, v), (x, w), (x, y), (u, y)\in E(\vG).$$
            Suppose $qv$ is oriented so that it is impossible to have $\vG[A]\cong \vT_3$. Then $\vG[B]\not\cong \vT_3$ as well. The probability that $\vG[C], \vG[D]$ or $\vG[E]\cong \vT_3$ is $1/2$. Thus suppose $qv$ is oriented in a way that we may have $\vG[A]\cong \vT_3$ or $\vG[B]\cong \vT_3$. Then the probability that $\vG[A]\cong \vT_3$ or $\vG[B]\cong \vT_3$ is $3/4$. If $\vG[A]\not\cong\vT_3$ and $\vG[B]\not\cong \vT_3$, then we have $\vG[E]\not\cong \vT_3$. The probability that $\vG[C]\cong \vT_3$ is $\frac{1}{2}$. If $\vG[C]\not\cong \vT_3$ then $\vG[D]\not\cong \vT_3$. Thus, the probability that at least one of $\vG[A], \vG[B], \vG[C], \vG[D]$ or $\vG[E]\cong \vT_3$ is 
            $$\frac{1}{2}\cdot \frac{1}{2} + \frac{1}{2}\cdot \frac{3}{4} + \frac{1}{2}\cdot\frac{1}{4}\cdot\frac{1}{2}= \frac{11}{16}.$$
        \end{enumerate}
        There are two orientations of $Y$ that fall into of case (i) and ten which fall into each of cases (ii), (iii), and (iv). Thus, since there are three possible vertices that could be $q$, the probability that neither $\vG[A], \vG[B], \vG[C]$ nor $\vG[D]\cong \vT_3$ for any $q = q_i$ is 
        $$\frac{1}{16}\left(\frac{1}{16}\right)^3 + \frac{5}{16}\left(\frac{3}{16}\right)^3 + \frac{5}{16}\left(\frac{1}{4}\right)^3 + \frac{5}{16}\left(\frac{5}{16}\right)^3 \approx 0.016 <\frac{1}{2^5}.$$
    \end{enumerate}
\end{proof}
\begin{proof}[Proof of Claim \ref{paths-claim}]
Enumerate the three triangles of $\vF_{3,3}$ by $\vT_1, \vT_2, \vT_3$. Let $\vG$ be an orientation of $G$ chosen uniformly at random. Consider two copies of $P_3$: $x_1y_1z_1\in G[X_1]$ and $x_2y_2z_2\in G[X_2]$ and suppose we also have edges $y_1x_2, y_1z_2, y_2x_1, y_2z_1$, and $y_1y_2$ in $G$. Let $Y = G[\{x_1, y_1, z_1, x_2, y_2, z_2\}]$.  We aim to find in $Y$ a copy of $\vT_1$ such that the center of $\vT_1$ is at either $y_1$ or $y_2$ and the other two vertices are in the opposite part. Note then that there are four such triangles (whose vertices are $y_1, y_2$ and one of $x_1, x_2, z_1, z_2$) in $Y$. Regardless of how the edge $y_1y_2$ is oriented, there is one orientation of the remaining two edges in each of these four triangles so that they form a copy of $\vT_1$.  Thus, the probability that none of the four triangles is isomorphic to $\vT_1$ is $(\frac{3}{4})^4$.  Now, suppose we have a third copy of $P_3$, $x'_1y'_1z'_1\in G[X_1]$, along with edges $y'_1x_2, y'_1z_2, y_2x'_1, y_2z'_1$, and $y'_1y_2$ in $G$. Call $Y' =G[\{x_1', y_1', z_1', x_2, y_2, z_2\}]$. Then the event that there is a such a copy of $\vT_1$ in $Y'$ is independent of the event that there is a copy of $\vT_1$ in $Y$. Then the probability that there is no copy of $\vT_1$ in $Y$ or $Y'$ is strictly at most than $(\frac{3}{4})^8<\frac{1}{4}$. 

Suppose without loss of generality that the vertices $y_1x_2y_2$ form a copy of $\vT_1$.  Then there are at least $\frac{n}{2}-5$ vertices $v\in V_2$ such that $y_1v, x_1v, z_1v$ are all edges. Let $u, v\in N_{V_2}(\{y_, x_1, z_1\})$. Then $y_1x_1u$ and $y_1z_1v$ are both triangles. Because $\vF_{3,3}$ is antidirected, no matter how the edges $y_1x_1$ and $y_1z_1$ are oriented, there is a way to extend these orientations to copies of $\vT_2$ and $\vT_3$ respectively. Since there are at least $n/5$ choices of the pair $u,v$, the probability that none of these pairs, along with $x_1,y_1, z_1$ form copies of $\vT_1, \vT_2$ both centered at $y_1$ is at most 
$$\epsilon\defeq \left(\frac{15}{16}\right)^{n/5}$$
Therefore, for large enough $n$, if we choose an orientation $\vG$ of $G$ at random, the probability that $\vG$ does not contain a $3$-fan is at most $(\frac{3}{4})^8 + \epsilon<\frac{1}{4}$.  Now, since our graph $G$ only contains a maximum of $\floor{\frac{n^2}{4}} + 8$ edges to begin with, $G$ can be oriented in fewer than 
$$\frac{1}{4}2^{\floor{\frac{n^2}{4}} + 8}= 2^{\floor{\frac{n^2}{4}} + 6}$$
ways, a contradiction. 
\end{proof}

\printbibliography

@inproceedings{Erdos1974,
  author={Paul L. Erdos},
  title={Some New Applications Of Probability Methods to Combinatorial Analysis and Graph Theory},
  booktitle={Proc. 5th southeast. Conf. Comb., Graph Theor., Comput.; Boca Raton 1974},
  year = {1974},
  pages={39-51}
}

@article{Alon2006 ,
author = {Alon, Noga and Yuster, Raphael}, 
title = {The number of orientations having no fixed tournament},
journal = {Combinatorica}, volume = {26},
year = {2006},
number = {},
pages = {1--16}
}

@article{botler2021,
title = {Counting orientations of graphs with no strongly connected tournaments},
journal = {Discrete Mathematics},
volume = {345},
number = {12},
pages = {113024},
year = {2022},
issn = {0012-365X},
doi = {https://doi.org/10.1016/j.disc.2022.113024},
url = {https://www.sciencedirect.com/science/article/pii/S0012365X22002308},
author = {Fábio Botler and Carlos Hoppen and Guilherme Oliveira Mota},

}

@article{bucic2023, 
title={Counting H-free orientations of graphs}, 
volume={174}, 
DOI={10.1017/S0305004122000147}, number={1}, journal={Mathematical Proceedings of the Cambridge Philosophical Society}, 
publisher={Cambridge University Press}, author={Buci\'c, Matija and Janzer, Oliver and Sudakov, Benny}, 
year={2023}, 
pages={79–95}}

@article{Erds1995ExtremalGF,
  title={Extremal Graphs for Intersecting Triangles},
  author={Paul Erd\H{o}s and Zolt{\'a}n F{\"u}redi and Ronald J. Gould and David S. Gunderson},
  journal={J. Comb. Theory, Ser. B},
  year={1995},
  volume={64},
  pages={89-100}
}

@article{BALOGH20041,
title = {The number of graphs without forbidden subgraphs},
journal = {Journal of Combinatorial Theory, Series B},
volume = {91},
number = {1},
pages = {1-24},
year = {2004},
issn = {0095-8956},
doi = {https://doi.org/10.1016/j.jctb.2003.08.001},
url = {https://www.sciencedirect.com/science/article/pii/S0095895603001023},
author = {József Balogh and Béla Bollobás and Miklós Simonovits},

}

@article{CHEN2003159,
title = {Extremal graphs for intersecting cliques},
journal = {Journal of Combinatorial Theory, Series B},
volume = {89},
number = {2},
pages = {159-171},
year = {2003},
issn = {0095-8956},
doi = {https://doi.org/10.1016/S0095-8956(03)00044-3},
url = {https://www.sciencedirect.com/science/article/pii/S0095895603000443},
author = {Guantao Chen and Ronald J. Gould and Florian Pfender and Bing Wei},
}

@article{YUAN2021,
author = {Yuan, Long-Tu},
title = {Extremal graphs for odd wheels},
journal = {Journal of Graph Theory},
volume = {98},
number = {4},
pages = {691-707},
doi = {https://doi.org/10.1002/jgt.22727},
url = {https://onlinelibrary.wiley.com/doi/abs/10.1002/jgt.22727},
eprint = {https://onlinelibrary.wiley.com/doi/pdf/10.1002/jgt.22727},
year = {2021}
}

\end{document}